\documentclass[reqno]{amsart}
\usepackage[utf8]{inputenc}
\usepackage{amssymb, amsmath, amsthm,bbm, color, enumerate}
\usepackage{mathtools}
\usepackage{enumitem}
\usepackage{csquotes}

\usepackage[backend=biber,style=alphabetic,dateabbrev=false,urldate=long, url=false]{biblatex}

\usepackage{xurl}   
\usepackage[bookmarksnumbered, hidelinks, colorlinks=false]{hyperref}
\usepackage[noabbrev, capitalize]{cleveref}

\newcounter{alltheorems}[section]

\theoremstyle{plain}
\newtheorem{theorem}[alltheorems]{Theorem}
\newtheorem{lemma}[alltheorems]{Lemma}
\newtheorem{corollary}[alltheorems]{Corollary}
\newtheorem{problem}[]{Open Problem}

\theoremstyle{definition}
\newtheorem{definition}[alltheorems]{Definition}

\theoremstyle{remark}
\newtheorem{remark}[]{Remark}

\newcommand{\R}{\mathbb{R}}
\newcommand{\C}{\mathbb{C}}
\DeclareMathOperator{\spt}{spt}
\DeclareMathOperator{\dens}{dens}
\DeclareMathOperator{\bd}{bd}
\DeclareMathOperator{\im}{im}
\DeclareMathOperator{\tp}{top}

\DeclareMathOperator{\diam}{diam}
\DeclareMathOperator{\sgn}{sgn}

\binoppenalty = \maxdimen
\relpenalty = \maxdimen

\bibliography{bib.bib}

\title[Maximal Polynomial Modulations of Singular Radon Transforms]{Maximal Polynomial Modulations of Singular Radon Transforms}
\author{Lars Becker}
\date{\today}

\subjclass[2020]{42B20}
\keywords{}
	
\address{Mathematical Institute, 
	University of Bonn,
	Endenicher Allee 60, 53115, Bonn,
	Germany. }
	\email{becker@math.uni-bonn.de}

\begin{document}

\begin{abstract}
    We prove $L^2 \to L^p$ estimates on the torus for maximal polynomial modulations of Calderón-Zygmund operators with anisotropic scaling. We obtain improved constants in these estimates. As a corollary,  maximal polynomial modulations of a mollified version of the Hilbert transform along the parabola are bounded with only logarithmic dependence of the estimate on the Lipschitz constant of the mollifier.
\end{abstract}

\maketitle

\section{Introduction}
Given a singular integral operator $T: L^p(\R^\mathbf{d}) \to L^p(\R^\mathbf{d})$ and a set $\mathcal{Q}$ of polynomial functions $\R^\mathbf{d} \to \R$, define the maximal modulation operator $T^\mathcal{Q}$ by
\[
    T^\mathcal{Q}f(x) = \sup_{Q \in \mathcal{Q}} |T(M_Q f)(x)|\,,
\]
where $M_Q f(t) = e^{iQ(t)}f(t)$. 

The study of maximal modulation operators grew out of the work of Carleson \cite{Carelson1966}, who proved that the Fourier series of an $L^2$ function $f$ converges pointwise almost everywhere to $f$.
Carleson's theorem is equivalent to the boundedness from $L^2(\R)$ into $L^{2, \infty}(\R)$ of the Carleson operator $H^{\mathcal{Q}_1} f$, where $\mathcal{Q}_1$ is the set of linear polynomials and $H$ denotes the Hilbert transform
\[
    Hf(x) =  p.v.\,\frac{1}{\pi}\int f(x-t) \frac{\mathrm{d}t}{t}\,.
\]
Later, other proofs of Carleson's theorem were given by Fefferman \cite{Fefferman1973} and Lacey and Thiele \cite{Lacey+2000}.
Sjölin \cite{Sjölin1971} replaced the Hilbert transform $H$ with a general Calderón-Zygmund operator $T$ and proved boundedness on $L^2(\R^\mathbf{d})$ of the maximal modulation operators $T^{\mathcal{Q}_1}$, see also \cite{Pramanik2003} for an alternative proof. A multilinear analogue of the Carleson operator was studied in \cite{Li2007}.

The investigation of more general maximal polynomial modulation operators originates in the work of Stein \cite{Stein1995}. He proved an $L^2(\R)$ estimate for the operator $H^\mathcal{Q}$, where $\mathcal{Q} = \{\alpha t^2 \, : \, \alpha \in \R\}$.
Stein and Wainger \cite{Stein+2001} showed $L^2(\R^\mathbf{d})$ bounds for $T^\mathcal{Q}$ for Calderón-Zygmund operators $T$ and $\mathcal{Q}$ the set of all polynomials of degree at most $d$ with no linear term. Finally, the restriction on the linear term of the polynomials was removed by Lie \cite{Lie2009}, \cite{Lie2020} for the Hilbert transform and by Zorin-Kranich \cite{ZK2021} for Hölder continuous Calderón-Zygmund kernels in arbitrary dimension.

This article deals with a problem introduced by Pierce and Young \cite{Pierce2019}. They considered singular integrals on a paraboloid
\[
    Sf(x,y) = \int f(x - z, y - |z|^2) K(z) \, \mathrm{d}z\,,
\]
where $K$ is some Calderón-Zygmund kernel on $\R^\mathbf{d}$, and proved $L^p(\R^{\mathbf{d}+1})$ estimates for the operator $S^\mathcal{Q}$ when $1 < p < \infty$, $\mathbf{d}\geq 2$ and $\mathcal{Q}$ is a subspace of polynomials satisfying certain restrictions. In particular, the polynomials in $\mathcal{Q}$ are not allowed to have linear terms and the quadratic term cannot be a multiple of $|x|^2$. This motivates the following question:
\begin{problem}
\label{mainproblem}
Define the Hilbert transform along the parabola by 
\[
    H_P f(x,y) = p.v.\,\int f(x - t, y - t^2) \, \frac{\mathrm{d}t}{t}\,.
\]
Is the maximally modulated Hilbert transform $H_P^{\mathcal{Q}_1}$ bounded from $L^2(\R^2)$ into $L^{2,\infty}(\R^2)$?
\end{problem}

A natural line of attack towards \cref{mainproblem} is to approximate the Hilbert transform along the parabola by Calderón-Zygmund operators $T$ with anisotropic scaling which have smooth kernels away from $0$. If one could prove uniform bounds for maximal modulations of such operators $T$, one could answer \cref{mainproblem} positively by a limiting argument. In the present paper, we prove bounds for maximal polynomial modulations of such operators $T$, extending the results from \cite{ZK2021} to the anisotropic setting, and improve the constant in the estimates compared to \cite{ZK2021}.

Fix $\mathbf{d} \in \mathbb{N}$ and a vector of exponents $\alpha = (\alpha_1, \dotsc, \alpha_\mathbf{d}) \in \mathbb{N}^\mathbf{d}$ with $\alpha_1 \leq \dotsc \leq \alpha_\mathbf{d}$. Denote its length by $|\alpha| = \sum_{i=1}^\mathbf{d} \alpha_i$. The anisotropic dilations with exponent $\alpha$ are 
\[
    \delta_r(x_1, \dotsc, x_\mathbf{d}) = (r^{\alpha_1} x_1, \dotsc, r^{\alpha_\mathbf{d}} x_\mathbf{d})\,
\]
and the anisotropic distance function is defined as
\[
    \rho(x) = \inf\{r > 0 \, : \, |\delta_{r^{-1}}(x)| \leq 1\}\,.
\]
A \emph{Calderón-Zygmund kernel with anisotropic scaling} is a tempered distribution $K$ agreeing on $\R^\mathbf{d} \setminus \{0\}$ with a function $K: \R^\mathbf{d} \setminus \{0\} \to \C$ such that for some $A > 0$
\begin{align}
    \label{KernelUpperBoundIntro}
    |K(x)| &\leq A \, \rho(x)^{-|\alpha|}\,,\\
    \label{KernelHölderBoundIntro}
    |K(x) - K(x')|  &\leq  A \, \frac{\rho(x-x')}{\rho(x)^{|\alpha| + 1}}\ \text{if}\ 2\rho(x-x') \leq \rho(x)\,,\\
    \label{kernelL2bound}
    |\hat K(\xi)| &\leq A\,.
\end{align}
The Calderón-Zygmund operator $T$ associated to $K$ is the operator defined on Schwartz functions $f$ by $Tf = K*f$. Note that \eqref{kernelL2bound} implies $\|Tf\|_2 \leq A\|f\|_2$ for $f \in \mathcal{S}(\R^\mathbf{d})$.

An \emph{admissible decomposition} of a Calderón-Zygmund kernel $K$ with anisotropic scaling is a decomposition $K = \sum_{s \in \mathbb{Z}} K_s$, where for all $s \in \mathbb{Z}$, the function $K_s$ is supported in $\{x \, : \, 2^{s-1}/4 \leq \rho(x) \leq 2^s/2\}$ and is Lipschitz continuous with constant $A 2^{-(|\alpha|+1)s}$. 
Given an admissible decomposition, we define for $f \in L^1(\mathbb{T}^\mathbf{d})$ the maximally truncated singular integral
\begin{align}
    R^Kf(x) = \sup_{\underline{\sigma} \leq \overline{\sigma}\leq 0} \left| \sum_{s = \underline{\sigma}}^{\overline{\sigma}} \int_{y \in x + \mathbb{T}^\mathbf{d}} K_s(x-y)f(y) \, \mathrm{d}y \right|
\end{align}
and the maximal average
\begin{align}
    \label{MaximalFctDef}
    M^K f(x) = \sup_{s} \int_{y \in x + \mathbb{T}^\mathbf{d}} |K_s(x-y)||f(y)| \, \mathrm{d}y\,.
\end{align}
It is not hard to explicitly construct admissible decompositions, however for proving estimates for $M^K$ and $R^K$ it is often helpful to have flexibility in the choice of $K_s$. 

Fix $d \in \mathbb{N}$ and let $\mathcal{Q}$ be the vector space of all real polynomials on $\R^{\mathbf{d}}$ of degree at most $d$ without constant term. We define for $f \in L^2(\mathbb{T}^\mathbf{d})$ 
\begin{align}
    \label{TTDef}
    T^\mathcal{Q} f(x) = \sup_{Q \in \mathcal{Q}} \sup_{0 < \underline{R} < \overline{R}} \left| \int_{y \in x + \mathbb{T}^\mathbf{d} :\underline{R} < \rho(x-y) < \overline{R}} K(x-y) e^{i Q(y)} f(y) \, \mathrm{d}y \right|\,.
\end{align}
Here and in all similar integrals we identify $\mathbb{T}^\mathbf{d}$ with $[-1/2, 1/2)^\mathbf{d}$.
Then we have the following new theorem, which extends the main result from \cite{ZK2021} to Calderón-Zygmund operators with anisotropic scaling and improves the constants.
\begin{theorem}
    \label{MainThmWeakBound1}
    Let $1 \leq p < 2$ and let $K$ be a Calderón-Zygmund kernel with anisotropic scaling, as defined in \eqref{KernelUpperBoundIntro} - \eqref{kernelL2bound} with constant $A$, with admissible decomposition $K = \sum_{s \in \mathbb{Z}} K_s$. 
    Then the operator $T^\mathcal{Q}$ defined by \eqref{TTDef} is bounded from $L^2(\mathbb{T}^\mathbf{d})$ into $L^p(\mathbb{T}^\mathbf{d})$ with  
    \begin{equation}
        \label{MainThmWeakBound1Eqn}
         \|T^\mathcal{Q}\|_{2 \to p} \lesssim_{\alpha, d,p} (\|R^K\|_{2 \to 2} + \|M^K\|_{2 \to 2}) \log^2(e + \frac{A}{\|M^K\|_{2 \to 2} + \|R^K\|_{2 \to 2}})\,.
    \end{equation}
\end{theorem}

\begin{remark}
    The conclusion of \cref{MainThmWeakBound1} still holds if $K$ is merely Hölder continuous. In fact, the same proof works for Hölder continuous kernels, one merely has to change some exponents. 
\end{remark}

We explain how this estimate \enquote{improves the constant}. It is possible to adapt the proof in \cite{ZK2021} to the anisotropic setting and to make all constants in the proof explicit, this yields the estimate $\|T^\mathcal{Q}\|_{2 \to 2} \lesssim A$, and thus  $\|T^\mathcal{Q}\|_{2 \to p} \lesssim A$.
To compare this to estimate \eqref{MainThmWeakBound1Eqn}, note that the function $x\log^2(e + A/x)$ is increasing in $x> 0$ for all $A > 0$, and that it holds in general that $\|M^K\|_{2 \to 2} + \|R^K\|_{2 \to 2} \lesssim A$. Thus \cref{MainThmWeakBound1} also implies the estimate $\|T^\mathcal{Q}\|_{2 \to p} \lesssim A$ for $p < 2$, and improves it if better estimates for $M^K$ and $R^K$ are available. 

We now give some examples where this is the case.
Studying \cref{mainproblem}, the situation we are mostly interested in is when $K$ approximates a Hilbert transform along a homogeneous curve. In this case, $R^K$ and $M^K$ are bounded uniformly, because the maximal functions associated to the Hilbert transform along a homogeneous curve are bounded, see \cite{Stein+1978}. Hence \eqref{MainThmWeakBound1Eqn} implies that in this case, the dependence of $\|T^\mathcal{Q}\|_{2 \to p}$ on the Lipschitz constant $A$ improves from linear to squared logarithmic. 
The same applies for singular integrals supported on paraboloids, as discussed in \cite{Pierce2019}.
We refer to \cite{Christ1999} for a general criterion for boundedness of the maximal functions associated to singular integrals on submanifolds. For all singular integrals supported on submanifolds for which the associated maximal average and maximally truncated operator are bounded, \eqref{MainThmWeakBound1Eqn} implies the same improvement of the constant from linear in the Lipschitz constant $A$ to squared logarithmic in $A$.

To answer \cref{mainproblem} positively, one would have to show uniform estimates for operators $T$ approximating a Hilbert transform along the parabola. Our improvement of the constant can thus be viewed as partial progress towards \cref{mainproblem}. In the Appendix we have worked out how this progress manifests itself in the roughness of the singular integrals for which one can show boundedness of maximal modulations.
We deduce from \cref{MainThmWeakBound1} that a version of \cref{mainproblem} can be answered positively where $H^P$ is replaced by a slightly less singular operator. While $H_P$ is supported on the parabola and thus looks like $\delta(y - x^2)$ on a line $x = const.$, our operator will have a singularity  $(|y-x^2| \log^{3+\varepsilon}(1/|y-x^2|))^{-1}$. In this context one should think of $\delta$ as having a singularity of order $-1$, so this misses the required singularity only by a logarithmic factor. As another application we show that maximal polynomial modulations of homogeneous, odd Calderón-Zygmund kernels with anisotropic scaling are bounded, under a very weak assumption on the modulus of continuity. All of these results are new.

We note that in all of our results, the only information used about $H_P$ is the boundedness of the maximal average and the maximally truncated operator associated to $H_P$. This is in contrast to the results in \cite{Pierce2019}, which use algebraic properties of the paraboloid. It would be interesting to find ways to exploit the algebraic structure of the parabola to make further progress on \cref{mainproblem}. 
One consequence of the algebraic structure is the following simple observation: For every polynomial $Q$ of two variables, there exists a polynomial $Q'$ of one variable whose coefficients depend only on $x$, $y$ and $Q$, such that $Q(x-t, y-t^2) = Q'(x-t)$. Thus, to answer \cref{mainproblem} positively, it would suffice to show a version of \cref{MainThmWeakBound1} with uniform constants under the assumption that the polynomials only depend on $x$. However, this assumption does not seem to allow for improvements in the conclusion of \cref{MainThmWeakBound1} or simplifications in the proof.

\Cref{mainproblem} has been investigated by other authors, we list some relevant references. Since $H_P$ is a convolution operator, there exists a Fourier multiplier $m_P$ such that $H_P f = (m_P \hat f){\check{}}$. The anisotropic dilation symmetry of $H_P$ implies that $m_P$ is homogeneous of degree zero with respect to an anisotropic scaling.
Roos \cite{Roos2019} made some progress on \cref{mainproblem} by showing that, if $m$ is homogeneous of degree zero with respect to an anisotropic scaling and sufficiently smooth, and $Tf = (m \hat f){\check{}}$, then $T^{\mathcal{Q}_1}$ is bounded from $L^2(\R^\mathbf{d})$ into $L^{2,\infty}(\R^\mathbf{d})$. Unfortunately, the multiplier $m_P$ is not smooth enough to directly apply Roos's result to it. We note that our \cref{MainThmWeakBound1} generalizes Roos's result to polynomial modulations, by the Hörmander-Mikhlin theorem. Another related paper is \cite{Ramos2021}, where the Fourier multiplier $m_P$ is restricted to lines and uniform bounds for maximal modulations of the resulting Fourier multipliers on $\R$ are shown. In \cite{Guo+2017} the authors obtain $L^p(\R^2)$ bounds for partial suprema of $H_P M_Q f$. Finally we refer to \cite{Mnatsakanyan2022}, where convergence of analogues of Fourier series for certain perturbations of the trigonometric system is proved, using an abstraction of the methods in \cite{ZK2021} in a similar spirit as in the present paper.

\subsection{Outline of the Proof}
We now give an overview of the proof of \cref{MainThmWeakBound1}. To show the $L^2(\mathbb{T}^\mathbf{d}) \to L^p(\mathbb{T}^\mathbf{d})$ estimate, we fix $\lambda$ and estimate the measure of the set $\{x \, : \, |T^\mathcal{Q} f(x)| > \lambda\}$. 
The first step of the argument, given in Section 2, is a discretization of $T^\mathcal{Q}$: The continuous truncation is replaced by a discrete truncation and the suprema are eliminated using stopping time functions. 
The next step is a decomposition of the operator. 
First, we decompose the operator according to so called tiles, which are localized on some $D$-adic cube and on which the stopping time functions are localized in a small subset of the space of polynomials $\mathcal{Q}$. 
Next, the set of tiles is organized into certain collections of tiles. Estimates for the parts of the operator corresponding to these collections are shown in Section 3 and Section 4. We give short sketches of the corresponding arguments in the beginning of these sections. All of this is combined in Section 5, where the main theorem
is derived from the results of Sections 2 to 4.
The basic structure of the proof is due to Charles Fefferman \cite{Fefferman1973}, it was adapted to polynomial modulations by Lie \cite{Lie2009}, \cite{Lie2020}. Our argument is based on the paper \cite{ZK2021}. 

In large parts of the proof we follow \cite{ZK2021}, with some straightforward modifications to translate the proof to the anisotropic setting. The whole proof in \cite{ZK2021} could be modified in this manner, yielding $L^p(\mathbb{T}^\mathbf{d})$ boundedness of $T^\mathcal{Q}$ for $1 <p<\infty$ with $\|T^\mathcal{Q}\|_{p \to p} \lesssim A$. However, we were not able to obtain the improved constants of \cref{MainThmWeakBound1} in all steps of the proof. Because of that, our proof is organized differently in some parts:
\begin{itemize}
    \item we use an exceptional set instead of the \enquote{stopping generations} in \cite{ZK2021}, Lemma 3.3. This simplifies the arguments in Section 3 compared to those in \cite{ZK2021}, and in particular makes it possible to obtain the improved constants there.
    \item we estimate boundary parts of trees (see Section 3.2) using an exceptional set argument.
\end{itemize}
This is done similarly in \cite{Lie2009}, \cite{Fefferman1973}, essentially we revert the changes in the argument that were introduced by Lie \cite{Lie2020} to show strong $L^2$ estimates without interpolation.

The new contributions in the present paper are that we adapt the proof to the anisotropic setting, and that we keep track of the constants in the proof of \cref{MainThmWeakBound1} and improve them by combining the estimates from \cite{ZK2021} with simpler estimates using the maximal functions $M^K$ and $R^K$.
The basic idea of this argument is as follows: In the proof of \cref{MainThmWeakBound1}, the operator is decomposed according to some parameter $n$. The different parts are then estimated with exponential decay in $n$, so that the proof can be completed by summing a geometric series. The constant in these exponentially decaying estimates is proportional to $A$. We show in addition different estimates without decay in $n$, using the maximal functions $M^K$ and $R^K$. The constants in these estimates are of size $\|M^K\|_{2\to2} + \|R^K\|_{2 \to 2}$. Using the better one of these estimates one obtains the logarithmic upper bound. The square of the logarithm arises because the first estimate actually has slightly slower than exponential decay in $n$. 

\subsection{Notation}
We use the notation $X \lesssim Y$ if there exists a constant $C$ such that $X \leq CY$, we write $X \gtrsim Y$ if $Y \lesssim X$ and $X \sim Y$ means that $X \lesssim Y$ and $Y \lesssim X$. If not stated otherwise, the implicit constant $C$ depends only on the exponents $\alpha$, the dimension $\mathbf{d}$ and the degree $d$ of the polynomials. Sometimes we use $C$ to explicitly denote a constant which only depends on $\alpha$, $\mathbf{d}$ and $d$, similarly $\varepsilon$ is a small positive number depending only on $\alpha$, $\mathbf{d}$ and $d$. The letter $A$ is used throughout to denote the constant of the Calderón-Zygmund kernel $K$, as defined in \eqref{KernelUpperBoundIntro}, \eqref{KernelHölderBoundIntro} and \eqref{kernelL2bound}.

$B_\rho(x,r) = \{y \in \R^\mathbf{d} \, : \, \rho(x-y) \leq r\}$ denotes a ball with respect to the anisotropic distance function $\rho$, while $B(x,r)$ denotes standard euclidean balls on $\R^\mathbf{d}$. Similarly, the diameter of a set with respect to the anisotropic distance function is abbreviated as $\diam_\rho E = \sup_{x,y \in E} \rho(x-y)$.

The letter $M$ always stands for the anisotropic version of the Hardy-Littlewood maximal function. The $q$-maximal function is defined as $M^q f = (M|f|^q)^{1/q}$. By the $L^p(\R^\mathbf{d})$ boundedness of $M$ for $p > 1$, the $q$-maximal function is bounded on $L^p(\mathbb{R}^\mathbf{d})$ for $p > q$.

We denote by $\mathcal{S}(\R^\mathbf{d})$ the space of Schwartz functions on $\R^\mathbf{d}$ and use the notation $e(x) = \exp(i x)$.

\subsection{Acknowledgements}
I would like to thank my advisor Christoph Thiele for many valuable discussions about the mathematics in this text, as well as numerous helpful suggestions regarding its write up.
I also thank the anonymous referee for numerous suggestions that helped improve the paper.
Finally, I am grateful to Pauline Dietrich, Jan Holstermann and Fabian Höfer, for reading an earlier version of this text and providing some small corrections.
The author was supported by the Collaborative Research Center 1060 funded by the Deutsche
Forschungsgemeinschaft (DFG, German Research Foundation) and the Hausdorff Center for
Mathematics, funded by the DFG under Germany’s Excellence Strategy - GZ 2047/1, ProjectID 390685813.

\section{Discretization}
\label{DecompositionSection}
In this section, we carry out some basic reductions and decompose the operator $T^\mathcal{Q}$. We first discretize the truncation in the definition of $T^\mathcal{Q}$ and replace it by a smooth truncation. Then we construct the collection of all tiles, and use it to decompose the operator into pieces $T_\mathfrak{p}$, where $\mathfrak{p}$ runs through the set of tiles. Finally, we organize the set of all tiles into antichains and forests.
We mostly follow Section 2 of \cite{ZK2021} and adapt it to the anisotropic setting. The main difference is that we do not use stopping generations. Instead we use an exceptional set, see \cref{FirstExceptionalSet}. This also simplifies the decomposition we obtain. 

We note that there is a small inconsistency in the notation: In the definition of an admissible decomposition the kernels $K_s$ have support in $\{\rho(x) \sim 2^s\}$, while in \eqref{KernelSupport} below the support is in $\{\rho(x) \sim D^s\}$. Given an admissible decomposition $K_s$ one can build kernels satisfying \eqref{KernelSupport} with $D = 2^l$ by taking sums of $l$ consecutive $K_s$. The maximal functions $M^K$ and $R^K$ for the latter decomposition are clearly bounded by those for the former, hence this change does not cause any problems. We have chosen to accept this inconsistency, because it allows us to state \cref{MainThmWeakBound1} without reference to the constant $D$, which is fixed only later.

\subsection{Discretization by Scale}
\label{DiscScaleSection}

Fix a large $D \in \mathbb{N}$ such that $D$ is a power of two. $D$ will be specified in the proof of \cref{ComparingPolBound}. 
Let $K$ be a kernel with admissible decomposition $K = \sum_k \tilde K_k$. Let $D = 2^l$ and define
\[
    K_s = \sum_{k = l(s-1)}^{ls - 1} \tilde{K}_k\,.
\]
Then $K_s$ satisfies
\begin{align}
    \label{KernelUpperBound}
    |K_s(x)| &\lesssim A D^{-s|\alpha|}\,,\\
    \label{KernelHoelderBound}
    |K_s(x) - K_s(x')| &\lesssim A \frac{\rho(x - x')}{D^{s(1 + |\alpha|)}}\,, \ \text{and}\\
    \label{KernelSupport}
    \spt K_s \subset \{x \, &: \, D^{s-1}/8 \leq \rho(x) \leq D^s/4\}\,.
\end{align}
It further holds that
\begin{align*}
    T^\mathcal{Q} f(x) \lesssim \sup_{Q \in \mathcal{Q}} \sup_{\underline{\sigma} \leq \overline{\sigma}\leq 0} \left| \sum_{s = \underline{\sigma}}^{\overline{\sigma}} \int_{x + \mathbb{T}^\mathbf{d}} K_s(x-y) e^{iQ(y)} f(y) \, \mathrm{d}y \right| +  M^Kf(x)\,.
\end{align*}
Thus, it suffices to estimate
\begin{equation}
\label{DiscreteOperator}
    \sup_{Q \in \mathcal{Q}} \sup_{\underline{\sigma} \leq \overline{\sigma}\leq 0} \left| \sum_{s = \underline{\sigma}}^{\overline{\sigma}} \int_{x + \mathbb{T}^\mathbf{d}} K_s(x-y) e^{iQ(y)} f(y) \, \mathrm{d}y \right|
\end{equation}
where $K_s$ satisfies the properties \eqref{KernelUpperBound} to \eqref{KernelSupport}.

By continuity in $Q$ of the integral in \eqref{DiscreteOperator}, we can restrict the supremum in $Q$ to a countable dense subset of $\mathcal{Q}$. By the monotone convergence theorem, we can further restrict the suprema to a finite set of polynomials $Q$ and scales $\underline{\sigma} \leq \overline{\sigma}$. Then the suprema become maxima, hence there exist measurable stopping time functions $Q: \R^\mathbf{d} \to \mathcal{Q}$, $\underline{\sigma}: \mathbb{R}^\mathbf{d} \to \mathbb{Z}$ and $\overline{\sigma}: \mathbb{R}^\mathbf{d} \to \mathbb{Z}$ with $\overline{\sigma} \geq \underline{\sigma}$, taking only finitely many values, such that the restricted supremum equals the absolute value of
\[
    T^{Q,\sigma}f(x) =   \sum_{\underline \sigma(x) \leq s \leq \overline \sigma(x) } \int_{x + \mathbb{T}^\mathbf{d}} K_s(x-y) e(Q_x(x) - Q_x(y)) f(y) \, \mathrm{d}y\,.
\]
We conclude that it is enough to show bounds for the operator $T^{Q,\sigma}$ which do not depend on $Q$ or $\sigma$. In the following, we fix $Q$ and $\sigma$ and set $s_{\min} = \min \underline{\sigma}$ and $s_{\max} = \max \overline{\sigma}$. 

\subsection{\texorpdfstring{$D$}{D}-adic Cubes}
\label{Dadicsection}
The anisotropic $D$-adic grid $\mathcal{D}$ is defined as the union
\[
    \mathcal{D} = \bigcup_{s \leq 0} \mathcal{D}_s
\]
where
\[
    \mathcal{D}_s = \{x + \delta_{D^s}([0,1)^\mathbf{d}) \, : \, x \in \delta_{D^s}(\mathbb{Z}^\mathbf{d}) \cap [0,1)^\mathbf{d}\}\,.
\]
The elements of $\mathcal{D}$ are called $D$-adic cubes. If $I\in\mathcal{D}_s$, then we say that $I$ has scale $s$ and write $s(I) = s$. Two $D$-adic cubes are either disjoint or one is contained in the other. Every $D$-adic cube $I$ of scale $s$ is contained in a unique $D$-adic cube $\hat I$ of scale $s+1$. Two $D$-adic cubes $I$ and $J$ are neighbours if they have the same scale and their closures intersect. We denote the union of $I$ and all its neighbours by $I^*$. A box $L$ is a cartesian product of intervals. Its center is denoted by $c(L)$, and we write 
\[
    rL = c(L) + \delta_r(L - c(L))
\]
for the box $L$ anisotropically dilated by a factor of $r$.

\subsection{Tiles}
Recall that $\mathcal{Q}$ is the vector space of all polynomials of degree at most $d$ on $\R^\mathbf{d}$ with vanishing constant term. We will decompose $\mathcal{Q}$ into \textit{uncertainty regions}. They play the role of the dyadic frequency intervals in the proofs \cite{Fefferman1973}, \cite{Lacey+2000}  of Carleson's theorem. In the setting of polynomial modulations, one has to choose the decomposition of the space of polynomials differently depending on the spatial location.

For a bounded subset $I \subset \R^\mathbf{d}$ with nonempty interior, define a norm $\| \cdot \|_I$ on $\mathcal{Q}$ by 
\[
    \|Q\|_I = \sup_{x, x' \in I} |Q(x) - Q(x')|\,.
\]
We will write $B_I(Q, r) = \{Q' \in \mathcal{Q} \, : \, \|Q' - Q\|_I \leq r \}$ for the closed balls with respect to this norm. 
The importance of the norms $\|\cdot\|_I$ lies in the fact that one can obtain upper bounds for oscillatory integrals with phase $Q$ over boxes $I$ with power decay in $\|Q\|_I$, see \cref{VanDerCorputA1}.

The following lemma allows us to compare the norms $\|\cdot\|_B$ for nested balls $B$:
\begin{lemma}
    \label{PolynomialBound}
    If $Q \in \mathcal{Q}$ and $B(x,r) \subset B(x, R) \subset \R^\mathbf{d}$ are euclidean balls then
    \[
        \|Q\|_{B(x,R)} \lesssim  (R/r)^d \|Q\|_{B(x,r)}
    \]
    and
    \[
        \|Q\|_{B(x,r)} \lesssim r/R \|Q\|_{B(x,R)}\,.
    \]
    If $B_\rho(x,r) \subset B_\rho(x,R) \subset \R^\mathbf{d}$ are anisotropic balls then
    \[
        \|Q\|_{B_\rho(x,R)} \lesssim (R/r)^{d \alpha_\mathbf{d}} \|Q\|_{B_\rho(x,r)}
    \]
    and
    \[
        \|Q\|_{B_\rho(x,r)} \lesssim (r/R)^{\alpha_1} \|Q\|_{B_\rho(x,R)}\,.
    \]
\end{lemma}

\begin{proof}
    We start with the first inequality. Applying a translation and dilation, we can assume $x = 0$, $r = 1$ and $\|Q\|_{B(x,1)} = 1$. Then we have to show that 
    \[
        \sup_{B(0,R)} |Q(x)| \lesssim R^d\,.
    \]
    The coefficients of $Q$ can be expressed as linear combinations of finitely many values of $Q$ in the unit ball, see e.g. \cite{Nicolaides1972}. Hence the coefficients are $\lesssim_{d, \mathbf{d}} 1$, which implies the estimate.  For the second inequality we assume similarly $x = 0$, $R = 1$ and $\|Q\|_{B(x, 1)} = 1$. Then we have to show
    \[
        \sup_{B(0,r)} |Q(x)| \lesssim r\,,
    \]
    which is true by the same argument. The estimates in the anisotropic setting follows from those in the isotropic setting and the inclusions $B_\rho(0, R) \subset B(0, R^{\alpha_\mathbf{d}})$ for $R \geq 1$ and $B_\rho(0,r) \subset B(0, r^{\alpha_1})$ for $r \leq 1$.
\end{proof}

\begin{corollary}[\cite{ZK2021}, Cor. 2.9]
    \label{ComparingPolBound}
    If $D$ is chosen sufficiently large then for every $I \in \mathcal{D}$ and $Q \in \mathcal{Q}$ it holds that
    \begin{align}
        \label{ComparingPolBoundEqn}
        \|Q\|_{\hat I} \geq 10^4 \|Q\|_I\,.
    \end{align}
\end{corollary}

\begin{proof}
    By precomposing $Q$ with an anisotropic dilation and a translation, we can assume that $\hat I = [0,1)^\mathbf{d}$. Then by \cref{PolynomialBound}
    \begin{align*}
        \|Q\|_I &\leq \|Q\|_{B(c(I), \mathbf{d}^{1/2}D^{-1}) } \lesssim D^{-1}  \|Q\|_{B(c(I), \mathbf{d}^{1/2})} \\
        &\leq D^{-1} \|Q\|_{B(c([0,1)^\mathbf{d}), 2\mathbf{d}^{1/2})}\lesssim D^{-1}  \|Q\|_{B(c([0,1)^\mathbf{d}), 1/2)} \leq D^{-1}\|Q\|_{[0,1)^\mathbf{d}}\,.
    \end{align*}
    Hence \eqref{ComparingPolBoundEqn} holds for sufficiently large $D$ depending only on $d$, $\mathbf{d}$.
\end{proof}

\begin{definition}[\cite{ZK2021}, Def. 2.11]
    A \textit{pair} $\mathfrak{p}$ consists of a $D$-adic cube $I_\mathfrak{p}$, called the \textit{spatial cube} of $\mathfrak{p}$, and a Borel-measurable subset $\mathcal{Q}(\mathfrak{p}) \subset \mathcal{Q}$, called the \textit{uncertainty region} of the pair. We denote by $s(\mathfrak{p}) = s(I_\mathfrak{p})$ the \textit{scale} of the pair.
\end{definition}

We are now ready to construct the collection of all tiles.
\begin{lemma}[\cite{ZK2021}, Lem. 2.12]
    \label{PairsLemma}
    There exist collections of pairs $\mathfrak{P}_I$, indexed by $I \in \mathcal{D}$ with $s_{\min} \leq s(I) \leq s_{\max}$, such that the following holds:
    \begin{enumerate}[label=(\arabic*)]
        \item \label{PairsLemma1} For every pair $\mathfrak{p} \in \mathfrak{P}_I$ there is a \textit{central polynomial} $Q_{\mathfrak{p}} \in \mathcal{Q}(\mathfrak{p})$ with
        \[
            B_I(Q_\mathfrak{p}, 0.2) \subset \mathcal{Q}(\mathfrak{p}) \subset B_I(Q_\mathfrak{p}, 1)\,
        \]
        \item \label{PairsLemma2} for each $I \in \mathcal{D}$ the uncertainty regions $\mathcal{Q}(\mathfrak{p})$, $\mathfrak{p} \in \mathfrak{P}_I$ form a disjoint cover of $\mathcal{Q}$
        \item \label{PairsLemma3} if $I \subset I'$, $\mathfrak{p} \in \mathfrak{P}_I$ and $\mathfrak{p}' \in \mathfrak{P}_{I'}$, then either $\mathcal{Q}(\mathfrak{p}) \cap \mathcal{Q}(\mathfrak{p}') = \emptyset$ or $\mathcal{Q}(\mathfrak{p}') \subset \mathcal{Q}(\mathfrak{p})$.
    \end{enumerate}
\end{lemma}

\begin{proof}
    For each $D$-adic cube $I \in \mathcal{D}$ choose a maximal set of polynomials $\mathcal{Q}_I \subset \mathcal{Q}$ with the property that $\|Q - Q'\|_I \geq 0.7$ for all $Q, Q' \in \mathcal{Q}_I$. 
    Then the balls $B_I(Q, 0.3)$, $Q \in \mathcal{Q}_I$ are disjoint and the balls $B_I(Q, 0.7)$, $Q \in \mathcal{Q}_I$ cover $\mathcal{Q}$ by maximality. 
    Pick an enumeration $\mathcal{Q}_I = \{Q_j \, :\, j \in \mathbb{N}\}$ and set 
    \[
    \tilde{\mathcal{Q}}(I, Q_j) = B_I(Q_j, 0.7) \setminus \bigcup_{k < j} \tilde{\mathcal{Q}}(I, Q_k) \setminus \bigcup_{k\neq j} B(Q_k, 0.3)\,.
    \]
    This yields a partition
    \[
        \mathcal{Q} = \bigcup_{Q \in \mathcal{Q}_I} \tilde{ \mathcal{Q}}(I, Q)
    \]
    of $\mathcal{Q}$ with $B_I(Q, 0.3) \subset \tilde{\mathcal{Q}}(I, Q) \subset B_I(Q, 0.7)$ for all $Q \in \mathcal{Q}_I$. 
    
    The partitions $\tilde{\mathcal{Q}}(I, Q)$ satisfy \ref{PairsLemma1} and \ref{PairsLemma2}. We now modify them in order of decreasing scale of $I$ to obtain condition \ref{PairsLemma3}. For $s = s_{\max}$, we set $\mathcal{Q}(I, Q) = \tilde{\mathcal{Q}}(I, Q)$. Now assume that we have constructed partitions $\mathcal{Q}(I, Q)$, $Q \in \mathcal{Q}_I$ of $\mathcal{Q}$ satisfying condition \ref{PairsLemma3} for all $I \in \mathcal{D}$ of scale larger than $s$. Let $I \in \mathcal{D}_s$. Define for $Q \in \mathcal{Q}_I$
    \[
        \mathcal{Q}(I, Q) = \bigcup_{\hat Q \in \mathcal{Q}_{\hat I} \cap \tilde{\mathcal{Q}}(I, Q)} \mathcal{Q}(\hat I, \hat Q)\,.
    \]
    Then the sets $\mathcal{Q}(I, Q)$, $Q \in \mathcal{Q}_I$ form a partition of $\mathcal{Q}$ and condition \ref{PairsLemma3} holds for this $I$. It remains to check that \ref{PairsLemma1} still holds, i.e. that $B_I(Q, 0.2) \subset \mathcal{Q}(I, Q) \subset B_I(Q, 1)$.
    Pick $\tilde Q \in \mathcal{Q}(I, Q)$. There is some $\hat Q \in \mathcal{Q}_{\hat I}\cap \tilde{\mathcal{Q}}(I, Q)$ with $\tilde{ Q} \in \mathcal{Q}(\hat I, \hat Q)$. Hence
    \begin{align*}
        \|\tilde Q - Q\|_I &\leq \|\hat Q - Q\|_I + \|\tilde Q - \hat Q\|_I\\
        &\leq \|\hat Q - Q\|_I + 10^{-4} \|\tilde Q - \hat Q\|_{\hat I}\leq  0.7 + 10^{-4} \leq 1\,.
    \end{align*}
    On the other hand, if $\tilde{Q} \notin \mathcal{Q}(I, Q)$, then there is some $\hat Q \in \mathcal{Q}_{\hat I}\setminus \tilde{\mathcal{Q}}(I, Q)$ with $\tilde{Q} \in \mathcal{Q}(\hat I, \hat Q)$. Therefore
    \begin{align*}
        \|\tilde Q - Q\|_I\geq \|\hat Q - Q\|_I - \|\tilde Q - \hat Q\|_I\geq \|\hat Q - Q\|_I - 10^{-4} \|\tilde Q - \hat Q\|_{\hat I} \geq 0.2\,.
    \end{align*}
    Thus, the partitions $\mathcal{Q}(I, Q)$ satisfy conditions \ref{PairsLemma1} to \ref{PairsLemma3}. 
\end{proof}

From now on, we fix collections of pairs $\mathfrak{P}_I$ as in \cref{PairsLemma} and define the set of all tiles
\[
    \mathfrak{P} = \bigcup_{I \in \mathcal{D}, s_{\min} \leq s(I)\leq s_{\max}} \mathfrak{P}_I\,.
\]
For each $\mathfrak{p} \in \mathfrak{P}$, we further fix a choice of central polynomial $Q_\mathfrak{p}$ such that condition \ref{PairsLemma1} of \cref{PairsLemma} holds.

Given a pair $\mathfrak{p}$, we define
\[
    E(\mathfrak{p}) = \{x \in I_\mathfrak{p}\, : \, Q_x \in \mathcal{Q}(\mathfrak{p}) \, \text{and} \, \underline \sigma(x) \leq s(\mathfrak{p}) \leq  \overline \sigma(x) \} 
\]
and 
\[
    \overline{E}(\mathfrak{p}) = \{x \in I_\mathfrak{p}\, : \, Q_x \in \mathcal{Q}(\mathfrak{p})\}\,.
\]
The larger sets $\overline{E}(\mathfrak{p})$ will be used in arguments exploiting the smallness of the density of tiles in the proof of \cref{AntichainsSeperation} and the proofs of \cref{SumTreeBound}. There it will be necessary to control the measure of the set of all $x \in I_\mathfrak{p}$ with $Q_x \in \mathcal{Q}(\mathfrak{p})$, regardless of $\sigma(x)$. 

The operator associated to the tile $\mathfrak{p}$ is defined as
\[  
    T_\mathfrak{p}f(x) = \mathbf{1}_{E(\mathfrak{p})}(x) \int_{x + \mathbb{T}^\mathbf{d}} K_{s(\mathfrak{p})}(x - y) e(Q_x(x) - Q_x(y)) f(y) \, \mathrm{d}y\,, 
\]
its adjoint is
\[
    T_\mathfrak{p}^*g(y) = \int_{y + \mathbb{T}^\mathbf{d}} e(-Q_x(x) + Q_x(y)) \overline{K_{s(\mathfrak{p})}(x -y)} (\mathbf{1}_{E(\mathfrak{p})}g)(x) \, \mathrm{d}x\,.
\]
Clearly $T_\mathfrak{p} f$ vanishes outside of $E(\mathfrak{p})$, and for all functions $g$ it holds that
\[
    \spt T_\mathfrak{p}^* g  \subset I_\mathfrak{p}^*\,.
\]
Given any subset $\mathfrak{S} \subset \mathfrak{P}$ we define
\[
    T_\mathfrak{S} = \sum_{\mathfrak{p} \in \mathfrak{S}} T_\mathfrak{p}\,.
\]
Then we have in particular that $T_\mathbb{T}^{Q, \sigma} = T_\mathfrak{P}$. 

\subsection{Organizing the Set of Tiles}
Now we further organize the set of all tiles into forests and antichains, following \cite{ZK2021}.

We define a partial order $\leq$ on the set of all tiles, similarly to the order introduced in \cite{Fefferman1973}. 
\begin{definition}
Let $\mathfrak{p}$, $\mathfrak{p}'$ be pairs. We say that 
\begin{itemize}
    \item $\mathfrak{p} < \mathfrak{p}'$ if $I_\mathfrak{p} \subsetneq I_{\mathfrak{p}'}$ and $\mathcal{Q}(\mathfrak{p}') \subset \mathcal{Q}(\mathfrak{p})$
    \item $\mathfrak{p} \leq \mathfrak{p}'$ if $I_\mathfrak{p} \subset I_{\mathfrak{p}'}$ and $\mathcal{Q}(\mathfrak{p}') \subset \mathcal{Q}(\mathfrak{p})$.
\end{itemize}
\end{definition}

\Cref{MainThmAnisotopic} and \cref{MainThmWeakBound1} are proven by estimating the distribution function $\lvert\{\lvert T_\mathfrak{P} f(x)\rvert > \lambda\}\rvert$ of $T_\mathfrak{P}f$, see Section 5. From now on we will therefore fix the parameter $\lambda > 10e$.
\begin{lemma}
    \label{FirstExceptionalSet}
    There exists an exceptional set $E_1 \subset \mathbb{T}^\mathbf{d}$ which is a disjoint union of $D$-adic cubes, such that the following holds:
    \begin{enumerate}[label=(\arabic*)]
         \item \label{FirstExceptionalSet1} $|E_1| \lesssim \lambda^{-2}$
         \item \label{FirstExceptionalSet2} Let 
         \[
            \mathfrak{P}_{good} = \{ \mathfrak{p} \in \mathfrak{P} \, : \,  I_\mathfrak{p} \not\subset E_1\}\,.
         \]
         For every $n \geq 1$ the set of tiles
         \[
            \mathfrak{M}_n = \{\mathfrak{p} \in \mathfrak{P}_{good} \, \text{maximal w.r.t. $\leq$ s.t.} \,  |\overline{E}(\mathfrak{p})|/|I_\mathfrak{p}| \geq 2^{-n}\}
         \]
         satisfies 
         \begin{equation}
            \label{MaxTilesOverlapEqn}
            \bigg\| \sum_{\mathfrak{p} \in \mathfrak{M}_n} \mathbf{1}_{I_\mathfrak{p}} \bigg\|_{\infty} \lesssim 2^n \log(n+1)\log(\lambda) \,. 
         \end{equation}
     \end{enumerate}
\end{lemma}

\begin{proof}
    Let $\tilde{\mathfrak{M}}_n$ be the collection of maximal tiles $\mathfrak{p} \in \mathfrak{P}$ with respect to the ordering $\leq$  satisfying $|\overline{E}(\mathfrak{p})|/|I_\mathfrak{p}| \geq 2^{-n}$.
    Then the sets $\overline{E}(\mathfrak{p})$, $\mathfrak{p} \in \tilde{\mathfrak{M}}_n$ are pairwise disjoint: If $x \in \overline{E}(\mathfrak{p}) \cap \overline{E}(\mathfrak{p}')$ then $x \in I_\mathfrak{p} \cap I_{\mathfrak{p}'}$, thus without loss of generality $I_\mathfrak{p} \subset I_{\mathfrak{p}'}$. But also $Q_x \in \mathcal{Q}(\mathfrak{p}) \cap \mathcal{Q}(\mathfrak{p}')$, which implies $\mathcal{Q}(\mathfrak{p}') \subset \mathcal{Q}(\mathfrak{p})$. Therefore $\mathfrak{p} \leq \mathfrak{p}'$ and by maximality $\mathfrak{p} = \mathfrak{p}'$. The disjointness implies the Carleson packing condition
    \[
        \sum_{\mathfrak{p} \in \tilde{\mathfrak{M}}_n \, : \, I_\mathfrak{p} \subset J} |I_\mathfrak{p}| \leq 2^{n} \sum_{\mathfrak{p} \in \tilde{\mathfrak{M}}_n \, : \, I_\mathfrak{p} \subset J} |\overline{E}(\mathfrak{p})| \leq 2^n|J|
    \]
    for all $J \in \mathcal{D}$.
    Let $C$ be a large constant to be fixed later. Consider
    \[
        E_1 = \bigcup_{n \geq 1} \{x \in [0,1)^\mathbf{d}\, : \, \sum_{\mathfrak{p} \in \tilde{\mathfrak{M}}_n} \mathbf{1}_{I_\mathfrak{p}}(x) \geq C2^n \log(n+1)\}\,.
    \]
    By the John-Nirenberg inequality it holds that
    \[
        |E_1| \leq e^2 \sum_{n \geq 1} \exp(-\frac{1}{2e} \frac{C2^n \log(n+1)}{2^n}) = e^2 \sum_{n \geq 1} (n+1)^{-C/(2e)} \lesssim  2^{-C/(2e)}\,.
    \]
    Choosing $C = 4e \log(\lambda)/\log(2)$, we obtain $|E_1| \lesssim \lambda^{-2}$. Furthermore, the set $E_1$ satisfies condition \ref{FirstExceptionalSet2}. Indeed, with the definition of $\mathfrak{M}_n$ there, it holds that $\mathfrak{M}_n \subset \tilde{\mathfrak{M}}_n$. Assume that there was some point $x$ with 
    \[
        \sum_{\mathfrak{p} \in \mathfrak{M}_n} \mathbf{1}_{I_\mathfrak{p}}(x) > C2^n \log(n+1)\,.
    \]
    Let $\mathfrak{p} \in \mathfrak{M}_n$ be a tile with minimal $I_\mathfrak{p}$ such that $x \in I_\mathfrak{p}$. Then the above estimate holds on $I_\mathfrak{p}$, and since $\mathfrak{M}_n \subset \tilde{\mathfrak{M}}_n$ it follows that $I_\mathfrak{p} \subset E_1$, a contradiction to $\mathfrak{p} \in \mathfrak{P}_{good}$. 
\end{proof}

Now we turn to organizing the set of all good tiles $\mathfrak{P}_{good}$ into forests and antichains. 

\begin{definition}[\cite{ZK2021}, Sec. 3.2]
    A collection $\mathfrak{A}$ of tiles is called an \textit{antichain} if no elements of $\mathfrak{A}$ are comparable under $<$. A collection $\mathfrak{C}$ of tiles is called \textit{convex} if $\mathfrak{p}, \mathfrak{p}' \in \mathfrak{C}$ and $\mathfrak{p} \leq \mathfrak{p}'' \leq  \mathfrak{p}'$ implies $\mathfrak{p}'' \in \mathfrak{C}$. A collection $\mathfrak{D}$ of tiles is called a \textit{down subset} if $\mathfrak{p} \in \mathfrak{D}$, $\mathfrak{p}' \leq \mathfrak{p}$ implies $\mathfrak{p}' \in \mathfrak{D}$. 
\end{definition}

Given a tile $\mathfrak{p}$, we denote by $a\mathfrak{p}$ the pair $(I_\mathfrak{p}, B_{I_\mathfrak{p}}(Q_\mathfrak{p}, a))$. 

\begin{definition}[\cite{ZK2021}, Def. 3.13]
    A \textit{tree} is a convex collection of tiles $\mathfrak{T} \subset \mathfrak{P}_{good}$ together with a \textit{top tile} $\tp \mathfrak{T}$ such that, for all $\mathfrak{p} \in \mathfrak{T}$, we have $4\mathfrak{p} < \tp \mathfrak{T}$. We call $Q_\mathfrak{T} = Q_{\tp\mathfrak{T}}$ the \textit{central polynomial} and $I_\mathfrak{T} = I_{\tp \mathfrak{T}}$ the \textit{spatial cube} of the tree $\mathfrak{T}$.
\end{definition}

This definition is chosen so that the phase $Q_x$ is \enquote{almost constant} on the tree in the sense that $\|Q_x - Q_\mathfrak{T}\|_{I_\mathfrak{p}} \lesssim 1$ if $x \in E(\mathfrak{p})$ and $\mathfrak{p} \in \mathfrak{T}$. If it was constant, then the contribution $T_\mathfrak{T}f$ of the tree would be bounded by a maximally truncated singular integral $R^K M_{-Q_\mathfrak{T}} f$. Since it is only almost constant, there is an error term which is however dominated by the maximal function $M^K$.

For a tile $\mathfrak{p}$ and $Q \in \mathcal{Q}$, we denote
\[
    \Delta(\mathfrak{p}, Q) = \|Q_\mathfrak{p} - Q\|_{I_\mathfrak{p}} + 1\,.
\]

\begin{definition}[\cite{ZK2021}, Def. 3.15]
    Two trees $\mathfrak{T}_1$, $\mathfrak{T}_2$ are called \textit{$\Delta$-separated} if 
    \begin{align*}
        \mathfrak{p}_1 \in \mathfrak{T}_1, I_{\mathfrak{p}_1} \subset I_{\mathfrak{T}_2} &\implies \Delta(\mathfrak{p}_1, Q_{\mathfrak{T}_2}) > \Delta \\
        \mathfrak{p}_2 \in \mathfrak{T}_2, I_{\mathfrak{p}_2} \subset I_{\mathfrak{T}_1} &\implies \Delta(\mathfrak{p}_2, Q_{\mathfrak{T}_1}) > \Delta\,.
    \end{align*}
\end{definition}

\begin{definition}
    An \textit{$L^\infty$-forest of level $n$} is a disjoint union $\mathfrak{F} = \cup_j \mathfrak{T}_j$ of  $2^{\gamma n}$-separated trees such that
    \begin{equation}
        \label{ForestEqn}
        \|\sum_{j} \mathbf{1}_{I_{\mathfrak{T}_j}}\|_\infty \lesssim 2^n \log(n+1)\log(\lambda)\,.
    \end{equation}
    Here $\gamma$ is a constant that will be fixed in the proof of \cref{ForestBound}.
\end{definition}

The constant $\gamma$ will be of size $\sim \log \log(\lambda)$, and $\lambda$ will always be large enough to ensure $\gamma > 1$. This is different from \cite{ZK2021}, where it is a fixed absolute constant. We also note that the factor $\log \lambda$ does not occur in \cite{ZK2021}.

\begin{definition}
    Define the \textit{density} of a tile $\mathfrak{p}$ as
    \begin{equation}
        \label{densedefinition}
        \dens(\mathfrak{p}) = \sup_{a \geq 2} a^{-\dim \mathcal{Q}} \sup_{\mathfrak{p}' \in \mathfrak{P}_{good} : a \mathfrak{p} \leq a \mathfrak{p}'} \frac{
        |\overline{E}(a\mathfrak{p}')|}{|I_{\mathfrak{p}'}|}\,.
    \end{equation}
    For a collection of tiles $\mathfrak{S} \subset \mathfrak{P}_{good}$ we set $\dens(\mathfrak{S}) = \sup_{\mathfrak{p} \in \mathfrak{S}} \dens(\mathfrak{p})$.
\end{definition}

Let 
\[
    \mathfrak{h}_n = \{ \mathfrak{p} \in \mathfrak{P}_{good} \, : \, \dens(\mathfrak{p}) > C 2^{-n}\}\,,
\]
for an absolute constant $C \geq 1$ which is chosen in the proof of \cref{DecompositionTreeAntichain}. Note that the sets $\mathfrak{h}_n$ are down subsets and hence convex: If $\mathfrak{p}_1 \in \mathfrak{h}_n$ then there exists $\mathfrak{p}' \in \mathfrak{P}_{good}$ and $a \geq 2$ with $a\mathfrak{p}_1 \leq a\mathfrak{p}'$ and 
$$
    a^{-\dim \mathcal{Q}} \frac{|\overline{E}(a\mathfrak{p}')|}{|I_{\mathfrak{p}'}|} > C 2^{-n}\,.
$$
If $\mathfrak{p}_2 < \mathfrak{p}_1$ then $a\mathfrak{p}_2 \leq a\mathfrak{p}_1 \leq a\mathfrak{p}'$, hence $\mathfrak{p}'$ also witnesses that $\dens(\mathfrak{p}_2) > C2^{-n}$ and thus $\mathfrak{p}_2 \in \mathfrak{h}_n$.
\begin{lemma}[\cite{ZK2021}, Prop. 3.22]
    \label{DecompositionTreeAntichain}
    For every $n \geq 1$, the set $\mathfrak{h}_n$ can be represented as the disjoint union of $O(\gamma n^2 + \gamma n \log\log \lambda)$ antichains and $O(n + \log\log \lambda)$ $L^\infty$-forests of level $n$.
\end{lemma}

\begin{proof}
    We first want to restrict to the simpler set 
    \[
        \mathfrak{C}_n = \{\mathfrak{p} \in \mathfrak{P}_{good} \, : \, \exists \mathfrak{m} \in \mathfrak{M}_n, 2\mathfrak{p} < 100\mathfrak{m}\}\,,
    \]
    where $\mathfrak{M}_n$ is the set defined in \cref{FirstExceptionalSet}.
    To this end, we show that $\mathfrak{h}_n \setminus \mathfrak{C}_n$ can be decomposed into at most $n$ antichains.
    It suffices to show that there exists no chain $\mathfrak{p}_0 < \dotsb < \mathfrak{p}_n$ in $\mathfrak{h}_n \setminus \mathfrak{C}_n$: In that case we can iteratively define antichains $\mathfrak{A}_i$, $1 \leq i \leq n$ by taking $\mathfrak{A}_j$ as the maximal (w.r.t. $\leq$) tiles in $\mathfrak{h}_n \setminus \mathfrak{C}_n\setminus (\cup_{i <j} \mathfrak{A}_i)$. These are clearly antichains, and if there is no chain of length $n+1$, then $\cup_i \mathfrak{A}_i = \mathfrak{h}_n \setminus \mathfrak{C}_n$.
    Assume that there was a chain $\mathfrak{p}_0 < \dotsb < \mathfrak{p}_n$. By the definition of $\mathfrak{h}_n$, there exists a tile $\mathfrak{p}' \in \mathfrak{P}_{good}$ and $a \geq 2$ with $a \mathfrak{p}_n \leq a \mathfrak{p}'$ and 
    \begin{equation}
        \label{CombProofEqn1}
        C 2^{-n} < a^{-\dim \mathcal{Q}} |\overline{E}(a \mathfrak{p}')|/ |I_{\mathfrak{p}'}|\,.
    \end{equation}
    We claim that the set $\mathcal{Q}(a \mathfrak{p}') = B_{I_\mathfrak{p}'}(Q_{\mathfrak{p}'}, a)$ can be covered with $\lesssim a^{\dim \mathcal{Q}}$ uncertainty regions $Q(\mathfrak{p}'')$ of tiles $\mathfrak{p}''$ satisfying $I_{\mathfrak{p}'} = I_{\mathfrak{p}''}$. Indeed, by \cref{PairsLemma}, $B_{I_{\mathfrak{p}'}}(Q_{\mathfrak{p}'}, a)$ is contained in the union of all such uncertainty regions it intersects. On the other hand, each of these uncertainty regions is contained in $B_{I_{\mathfrak{p}'}}(Q_{\mathfrak{p}'},  a + 2)$. Furthermore, all of the uncertainty regions are disjoint and contain a ball in the norm $\| \cdot \|_{I_{\mathfrak{p}'}}$ of radius $0.2$. Therefore, one needs at most $(5(a + 2))^{\dim \mathcal{Q}}$ uncertainty regions $\mathcal{Q}(\mathfrak{p}'')$ to cover $B_{I_{\mathfrak{p}'}}(Q_{\mathfrak{p}'},a)$, as claimed.
    
    If  $x \in \overline{E}(a \mathfrak{p}')$, then $x \in I_{\mathfrak{p}'}$ and $Q_x \in \mathcal{Q}(a \mathfrak{p}')$. Thus
    \[
        \overline{E}(a \mathfrak{p}') \subset \bigcup \overline{E}(\mathfrak{p}'')\,,
    \]
    which implies by \eqref{CombProofEqn1} that one of the tiles $\mathfrak{p}''$ satisfies $C 2^{-n} \lesssim  |\overline{E}(\mathfrak{p}'')|/ |I_{\mathfrak{p}''}|$.
    We can now choose $C$ sufficiently large so that it satisfies $|I_{\mathfrak{p}''}| 2^{-n} \leq |\overline{E}(\mathfrak{p}'')|$.
    Then there exists, by \cref{FirstExceptionalSet}, some $\mathfrak{m} \in \mathfrak{M}_n$ with $\mathfrak{p}'' \leq \mathfrak{m}$.
    Equation \eqref{CombProofEqn1} also implies that
    \[
        a \leq C a^{\dim \mathcal{Q}} \leq 2^n |\overline{E}(a \mathfrak{p}')|/ |I_{\mathfrak{p}'}| \leq 2^n\,.
    \]
    Hence we have for all $Q \in \mathcal{Q}(100 \mathfrak{m})$ that
    \begin{align*}
        \|Q_{\mathfrak{p}_0} - Q\|_{I_{\mathfrak{p}_0}}
        &\leq \|Q_{\mathfrak{p}_0} - Q_{\mathfrak{p}_n}\|_{I_{\mathfrak{p}_0}} 
        + \|Q_{\mathfrak{p}_n} - Q_{\mathfrak{p}'}\|_{I_{\mathfrak{p}_0}} 
        + \|Q_{\mathfrak{p}'} - Q_{\mathfrak{p}''}\|_{I_{\mathfrak{p}_0}} \\
        &\quad+ \|Q_{\mathfrak{p}''} - Q_{\mathfrak{m}}\|_{I_{\mathfrak{p}_0}} 
        + \|Q_{\mathfrak{m}} - Q\|_{I_{\mathfrak{p}_0}} \\
        &\leq 1 + 10^{-4n}(\|Q_{\mathfrak{p}_n} - Q_{\mathfrak{p}'}\|_{I_{\mathfrak{p}_n}} 
        + \|Q_{\mathfrak{p}'} - Q_{\mathfrak{p}''}\|_{I_{\mathfrak{p}'}} 
        + \|Q_{\mathfrak{p}''} - Q_{\mathfrak{m}}\|_{I_{\mathfrak{p}''}}\\
        &\quad+ \|Q_{\mathfrak{m}} - Q\|_{I_{\mathfrak{m}}})\\
        &\leq 1 + 10^{-4n}(a + (a + 1) + 1+ 100) \leq 2\,.
    \end{align*}
    Here we used in the second step that $Q_{\mathfrak{p}_n} \in \mathcal{Q}(\mathfrak{p}_0) \subset B_{I_{\mathfrak{p}_0}}(Q_{\mathfrak{p}_0},1)$ which holds since $\mathfrak{p}_0 < \mathfrak{p}_n$,  we used \cref{ComparingPolBound} $n$ times, and we used that $I_{\mathfrak{p}_n} \subset I_{\mathfrak{p}'} = I_{\mathfrak{p}''} \subset I_{\mathfrak{m}}$. The third step follows from the relations $a \mathfrak{p}_n \leq a \mathfrak{p}'$, $Q_{\mathfrak{p}''} \in B_{I_{\mathfrak{p}'}}(Q_{\mathfrak{p}'}, a + 1)$, $\mathfrak{p}'' \leq \mathfrak{m}$, and the assumption $Q \in \mathcal{Q}(100\mathfrak{m})$. Hence $\mathcal{Q}(100\mathfrak{m}) \subset \mathcal{Q}(2 \mathfrak{p}_0)$. Since $I_{\mathfrak{p}_0} \subset I_{\mathfrak{m}}$, we can conclude that $2\mathfrak{p}_0 < 100 \mathfrak{m}$. This is the desired contradiction to $\mathfrak{p}_0 \in \mathfrak{h}_n \setminus \mathfrak{C}_n$.
    
    It remains to decompose the set $\mathfrak{h}_n \cap \mathfrak{C}_n$ into $O(\gamma n^2 + \gamma n \log\log \lambda)$ antichains and $O(n + \log\log \lambda)$ $L^\infty$-forests of level $n$. Since $\mathfrak{h}_n$ is convex, it suffices to decompose $\mathfrak{C}_n$: Intersecting the resulting trees and antichains with $\mathfrak{h}_n$ yields a decomposition of $\mathfrak{h}_n \cap \mathfrak{C}_n$. Let for $\mathfrak{p} \in \mathfrak{C}_n$
    \[
        \mathfrak{B}(\mathfrak{p}) = \{\mathfrak{m} \in \mathfrak{M}_n \, : \, 100 \mathfrak{p} \leq \mathfrak{m}\}\,.
    \]
    By the definition of $\mathfrak{M}_n$, there exists for each $\mathfrak{p} \in \mathfrak{P}_{good}$ some $\mathfrak{m} \in \mathfrak{M}_n$ with $100\mathfrak{p} \leq \mathfrak{p} \leq \mathfrak{m}$.  Furthermore, by \eqref{MaxTilesOverlapEqn}, the spatial cubes $I_\mathfrak{m}$, $\mathfrak{m} \in \mathfrak{M}_n$ have overlap bounded by $2^n \log(n+1) \log(\lambda)$. Together this implies
    \[
        1 \leq |\mathfrak{B}(\mathfrak{p})| \lesssim 2^n \log(n+1) \log(\lambda)
    \]
    for all $\mathfrak{p} \in \mathfrak{C}_n$. For $0 \leq j \lesssim n + \log\log(\lambda)$ let 
    \[
        \mathfrak{C}_{n, j} = \{\mathfrak{p} \in \mathfrak{C}_n \, : \, 2^j \leq |\mathfrak{B}(\mathfrak{p})| < 2^{j+1}\}\,.
    \]
    To complete the proof, we show that each set $\mathfrak{C}_{n,j}$ can be written as the union of one $L^\infty$-forest and $\lceil\gamma n\rceil+1$ antichains. 
    
    The set $\mathfrak{C}_{n,j}$ is convex: If $\mathfrak{p}_1 < \mathfrak{p}_2$, then $\hat I_{\mathfrak{p}_1} \subset I_{\mathfrak{p}_2}$ and $Q_{\mathfrak{p}_2} \in B_{I_{\mathfrak{p}_1}}(Q_{\mathfrak{p}_1}, 1)$. By \cref{ComparingPolBound}, this implies $B_{I_{\mathfrak{p}_2}}(Q_{\mathfrak{p}_2}, 100) \subset B_{I_{\mathfrak{p}_1}}(Q_{\mathfrak{p}_1}, 100)$ and thus $100 \mathfrak{p}_1 < 100 \mathfrak{p}_2$. Therefore for tiles $\mathfrak{p}_1 < \mathfrak{p} < \mathfrak{p}_2$ it holds that $\mathfrak{B}(\mathfrak{p}_2) \subset \mathfrak{B}(\mathfrak{p}) \subset \mathfrak{B}(\mathfrak{p}_1)$ and hence $\mathfrak{p}_1, \mathfrak{p}_2 \in \mathfrak{C}_{n.j}$ implies $\mathfrak{p} \in \mathfrak{C}_{n,j}$.
    
    We choose the tree tops for the $L^\infty$-forest. Let $\mathfrak{U} \subset \mathfrak{C}_{n,j}$ be the set of all tiles $\mathfrak{u}$ such that 
    \[
    \mathfrak{p} \in \mathfrak{C}_{n,j},I_\mathfrak{u} \subsetneq I_\mathfrak{p}\implies\mathcal{Q}(100 \mathfrak{u}) \cap \mathcal{Q}(100 \mathfrak{p}) = \emptyset\,.
    \] 
    We show that the spatial cubes of these tiles have overlap $\leq C2^n \log(n+1)\log(\lambda)$, i.e. the sets $\mathfrak{U}(x) = \{\mathfrak{u} \in \mathfrak{U} \, : \, x \in I_{\mathfrak{u}}\}$ have cardinality $\leq C2^n\log(n+1)\log(\lambda)$ for all $x$. 
    To show this estimate, we first write $\mathfrak{U}(x)$ as the union of $O(1)$ many collections $\mathfrak{U}'(x)$ such that $\mathcal{Q}(100 \mathfrak{u}) \cap \mathcal{Q}(100 \mathfrak{u}') = \emptyset$ for all $\mathfrak{u} \neq \mathfrak{u}' \in \mathfrak{U}'(x)$. This is possible at each fixed scale $s$: The uncertainty regions $\mathcal{Q}(\mathfrak{u})$, $\mathfrak{u} \in \mathfrak{U}'(x)$ with $s(\mathfrak{u}) = s$ are pairwise disjoint since all such $\mathfrak{u}$ have the same spatial cube. A ball $B_{I_\mathfrak{u}}(Q_\mathfrak{u}, 200)$ intersects at most $(202/0.2)^{\dim \mathcal{Q}}$ of these uncertainty regions $\mathcal{Q}(\mathfrak{u}')$, $\mathfrak{u}'\in \mathfrak{U}(x)$, thus each of the sets $Q(100 \mathfrak{u}) = B_{I_\mathfrak{u}}(Q_\mathfrak{u}, 100)$ intersects at most $2000^{\dim \mathcal{Q}}$ other sets $Q(100\mathfrak{u}')$, $\mathfrak{u}' \in \mathfrak{U}(x)$. Hence the claimed decomposition is possible at scale $s$. But if $\mathfrak{u}, \mathfrak{u}' \in \mathfrak{U}(x)$ are of different scale, then without loss of generality $I_\mathfrak{u} \subsetneq I_{\mathfrak{u}'}$ and thus, by the definition of $\mathfrak{U}$, the sets $Q(100\mathfrak{u})$, $Q(100 \mathfrak{u}')$ are disjoint. It follows that the claimed decomposition of $\mathfrak{U}(x)$ into $O(1)$ sets $\mathfrak{U}'(x)$ exists. 
    Since the sets $\mathcal{Q}(100\mathfrak{u})$, $\mathfrak{u} \in \mathfrak{U}'(x)$ are pairwise disjoint, the same is true for the sets $\mathfrak{B}(\mathfrak{u})$, $\mathfrak{u} \in \mathfrak{U}'(x)$. Each of these sets has cardinality $\geq 2^j$ and their union has cardinality $\lesssim 2^n \log(n+1) \log(\lambda)$ by \eqref{MaxTilesOverlapEqn}. Therefore $|\mathfrak{U}(x)| \lesssim 2^{n-j} \log(n+1) \log(\lambda)$, as required. 
    
    Next, we construct the trees belonging to the tree tops. For $\mathfrak{u} \in \mathfrak{U}$ let 
    \[
        \mathfrak{D}(\mathfrak{u}) = \{\mathfrak{p} \in \mathfrak{C}_{n,j} \, : \, 2\mathfrak{p} < \mathfrak{u}\}\,.
    \]
    Then the set
    \[
        \mathfrak{A}_j = \mathfrak{C}_{n,j} \setminus \bigcup_{u \in \mathfrak{U}} \mathfrak{D}(\mathfrak{u})
    \]
    is an antichain. Indeed, assume that there exists $\mathfrak{p} < \mathfrak{p}_1 \in \mathfrak{A}_j$. Using \cref{ComparingPolBound}, this implies $2\mathfrak{p} < 200 \mathfrak{p}_1 \leq \mathfrak{p}_1$. If $\mathfrak{p}_1 \in \mathfrak{U}$ then $\mathfrak{p} \in \mathfrak{D}(\mathfrak{p}_1)$, which contradicts $\mathfrak{p} \in \mathfrak{A}_j$. 
    Hence $\mathfrak{p}_1 \notin \mathfrak{U}$ and there exists $\mathfrak{p}_2 \in \mathfrak{C}_{n,j}$ with $I_{\mathfrak{p}_1} \subsetneq I_{\mathfrak{p}_2}$ and $\mathcal{Q}(100 \mathfrak{p}_1) \cap \mathcal{Q}(100 \mathfrak{p}_2) \neq \emptyset$. By \cref{ComparingPolBound}, it follows that $200 \mathfrak{p}_1 < 200 \mathfrak{p}_2$. Then $2\mathfrak{p} < \mathfrak{p}_2$. Continuing in this manner we obtain a chain $2\mathfrak{p} < 200\mathfrak{p}_1 < \dotsb < 200\mathfrak{p}_k$ of arbitrary length $k$ with $\mathfrak{p}_1, \dotsc, \mathfrak{p}_k \in \mathfrak{C}_{n,j}$. This is a contradiction as the tiles in $\mathfrak{P}$ have bounded scale.
    
    Let $\mathfrak{U}' = \{\mathfrak{u} \in \mathfrak{U} \, : \, \mathfrak{D}(\mathfrak{u}) \neq \emptyset\}$ and consider the relation 
    \[
        \mathfrak{u} \sim \mathfrak{u}' \iff \exists \mathfrak{p} \in \mathfrak{D}(\mathfrak{u}) \, \text{with} \, 10\mathfrak{p} \leq \mathfrak{u}'
    \]
    on $\mathfrak{U}'$, we will show below that it is an equivalence relation. We first claim that 
    \[
        \mathfrak{u} \sim \mathfrak{u}' \implies I_\mathfrak{u} = I_{\mathfrak{u}'} \, \text{and} \, \mathcal{Q}(100 \mathfrak{u}) \cap \mathcal{Q}(100 \mathfrak{u}') \neq \emptyset\,.
    \]
    Let $\mathfrak{u}, \mathfrak{u}' \in \mathfrak{U}'$ with $\mathfrak{u} \sim \mathfrak{u}'$. By definition of $\sim$ and $\mathfrak{D}(\mathfrak{u})$, there exists $\mathfrak{p}$ with $2\mathfrak{p} < \mathfrak{u}$ and $10\mathfrak{p} \leq \mathfrak{u}'$. Hence the cubes $I_\mathfrak{u}$ and $I_{\mathfrak{u}'}$ both contain $I_\mathfrak{p}$ and are therefore nested. Because of that, it suffices to show
    \begin{equation}
    \label{eq 100 intersection}
        \mathcal{Q}(100\mathfrak{u}) \cap \mathcal{Q}(100\mathfrak{u}') \neq \emptyset
    \end{equation}
    since this already implies $I_\mathfrak{u} = I_{\mathfrak{u}'}$ by the definition of $\mathfrak{U}$ and nestedness.
    Now we distinguish three cases. If $I_\mathfrak{p} = I_{\mathfrak{u}'}$, then  $100\mathfrak{u}' \leq 2\mathfrak{p} < \mathfrak{u}$, which gives \eqref{eq 100 intersection}. If $I_\mathfrak{p} = I_{\mathfrak{u}}$, then $100\mathfrak{u} \leq 10\mathfrak{p} \leq \mathfrak{u'}$, which gives again \eqref{eq 100 intersection}.
    If $I_\mathfrak{p} \subsetneq I_{\mathfrak{u}'}$ and $I_\mathfrak{p} \subsetneq I_{\mathfrak{u}}$, we deduce using \cref{ComparingPolBound} that  $100\mathfrak{p} < 100\mathfrak{u}'$ and $100\mathfrak{p} < 100\mathfrak{u}$. If $\mathcal{Q}(100\mathfrak{u}) \cap \mathcal{Q}(100\mathfrak{u}') = \emptyset$, then the sets $\mathfrak{B}(\mathfrak{u})$ and $\mathfrak{B}(\mathfrak{u}')$ are disjoint. However, $\mathfrak{B}(\mathfrak{p}) \supset \mathfrak{B}(\mathfrak{u}) \cup \mathfrak{B}(\mathfrak{u}')$ which implies $|\mathfrak{B}(\mathfrak{p})| \geq |\mathfrak{B}(\mathfrak{u})| + |\mathfrak{B}(\mathfrak{u}')| \geq 2^{j+1}$, contradicting $\mathfrak{p} \in \mathfrak{C}_{n,j}$. Thus $\mathcal{Q}(100\mathfrak{u}) \cap \mathcal{Q}(100\mathfrak{u}') \neq \emptyset$, as claimed.
    
    Next, we check that $\sim$ is an equivalence relation. Reflexivity is obvious. For symmetry and transitivity assume that $\mathfrak{u}, \mathfrak{u}', \mathfrak{u''} \in \mathfrak{U}'$ with $I_\mathfrak{u} = I_{\mathfrak{u}'} =  I_{\mathfrak{u}''}$ and $\mathcal{Q}(100\mathfrak{u}) \cap\mathcal{Q}(100\mathfrak{u}') \neq \emptyset$ and $\mathcal{Q}(100\mathfrak{u}') \cap\mathcal{Q}(100\mathfrak{u}'') \neq \emptyset$. Since $\mathfrak{D}(\mathfrak{u}) \neq \emptyset$, there exists $\mathfrak{p} \in \mathfrak{D}(\mathfrak{u})$. By definition, $\mathfrak{p}$ satisfies $2\mathfrak{p} < \mathfrak{u}$. Using \cref{ComparingPolBound}, it follows that $4\mathfrak{p} < 1000\mathfrak{u}$ and hence $4\mathfrak{p} < \mathfrak{u}''$.
    We conclude transitivity: If $\mathfrak{u} \sim \mathfrak{u}' \sim \mathfrak{u}''$, then our assumption holds and thus $10 \mathfrak{p} \leq 4\mathfrak{p} < \mathfrak{u}''$, which gives $\mathfrak{u} \sim \mathfrak{u}''$. Similarly symmetry: If $\mathfrak{u}' \sim \mathfrak{u}$, then the assumption holds with $\mathfrak{u}'' = \mathfrak{u}'$. Hence $10 \mathfrak{p} < \mathfrak{u}'$, which implies $\mathfrak{u} \sim \mathfrak{u}'$.
    
    Let $\mathfrak{V}$ be a set of representatives of $\mathfrak{U}'$ for $\sim$. For $\mathfrak{v} \in \mathfrak{V}$ define
    \[
        \mathfrak{T}(\mathfrak{v}) = \bigcup_{\mathfrak{u} \sim \mathfrak{v}} \mathfrak{D}(\mathfrak{u})\,.
    \]
    The sets $\mathfrak{T}(\mathfrak{v})$ are convex: Suppose that $\mathfrak{u}_1 \sim \mathfrak{v}$, $\mathfrak{u}_2 \sim \mathfrak{v}$, that $\mathfrak{p}_1 \in \mathfrak{D}(\mathfrak{u}_1)$, $\mathfrak{p}_2 \in \mathfrak{D}( \mathfrak{u}_2)$ and that $\mathfrak{p}_2 < \mathfrak{p} < \mathfrak{p}_1$. Then $\mathfrak{p} \in \mathfrak{C}_{n,j}$ by convexity of $\mathfrak{C}_{n,j}$, and $2\mathfrak{p} \leq 2\mathfrak{p}_1 \leq \mathfrak{u}_1$. Thus $\mathfrak{p} \in \mathfrak{D}(\mathfrak{u}_1) \subset \mathfrak{T}(\mathfrak{v})$. As we have shown above, for every $\mathfrak{p} \in \mathfrak{D}(\mathfrak{u})$ with $\mathfrak{u} \sim \mathfrak{v}$, it holds that $4\mathfrak{p} < \mathfrak{v}$. Therefore $\mathfrak{T}(\mathfrak{v})$ is a tree with top $\mathfrak{v}$. By the definition of $\sim$, these trees satisfy the separation condition
    \begin{equation}
        \label{SeparationCondition}
        \mathfrak{v} \neq \mathfrak{v}' \implies \forall \mathfrak{p} \in \mathfrak{T}(\mathfrak{v}): \ 10\mathfrak{p} \not \leq \mathfrak{v}'\,.
    \end{equation}
    As we have already shown the overlap estimate \eqref{ForestEqn}, it remains to obtain $2^{\gamma n}$-separation of the trees $\mathfrak{T}(\mathfrak{v})$, $\mathfrak{v} \in \mathfrak{V}$. In order to achieve this, we remove the bottom $\gamma n$ layers of tiles. Let for $k = 1, \dotsc, \lceil\gamma n\rceil$ 
    \[
    \mathfrak{A}'_{j,k} = \bigcup_{\mathfrak{v} \in \mathfrak{V}} \mathfrak{A}'_{j,k} (\mathfrak{v})\,,
    \]
    where $\mathfrak{A}'_{j,k}(\mathfrak{v})$ is the set of minimal tiles in $\mathfrak{T}(\mathfrak{v}) \setminus \cup_{i  < k} \mathfrak{A}_{j,i}'$.
    Tiles in $\mathfrak{A}'_{j,k}(\mathfrak{v})$ are clearly not comparable.
    If $\mathfrak{p} \in \mathfrak{A}'_{j,k}(\mathfrak{v}) , \mathfrak{p}' \in \mathfrak{A}'_{j,k}(\mathfrak{v}')$ are in different trees and $\mathfrak{p} \leq \mathfrak{p}'$, then since $4\mathfrak{p}' < \mathfrak{v}'$ it holds that $I_{\mathfrak{p}} \subset I_{\mathfrak{v}'}$. With the separation condition \eqref{SeparationCondition}, this implies that $\mathcal{Q}(\mathfrak{v}') \not\subset \mathcal{Q}(10\mathfrak{p})$. On the other hand, $4\mathfrak{p}' < \mathfrak{v}'$ implies that for $Q \in \mathcal{Q}(\mathfrak{v}')$ we have
    \begin{align*}
        \|Q - Q_\mathfrak{p}\|_{I_\mathfrak{p}} &\leq \|Q - Q_{\mathfrak{v}'}\|_{I_\mathfrak{p}} + \|Q_{\mathfrak{v}'} - Q_{\mathfrak{p}'}\|_{I_\mathfrak{p}} +
        \|Q_{\mathfrak{p}'} - Q_{\mathfrak{p}}\|_{I_{\mathfrak{p}}}\\
        &\leq \|Q - Q_{\mathfrak{v}'}\|_{I_{\mathfrak{v}'}} + \|Q_{\mathfrak{v}'} - Q_{\mathfrak{p}'}\|_{I_{\mathfrak{p}'}} +  \|Q_{\mathfrak{p}'} - Q_{\mathfrak{p}}\|_{I_{\mathfrak{p}}}\\
        &\leq 1 + 4 + 1 < 10\,,
    \end{align*}
    and thus $\mathcal{Q}(\mathfrak{v}') \subset \mathcal{Q}(10\mathfrak{p})$.
    This is a contradiction, and therefore tiles in distinct trees are not comparable. Thus $\mathfrak{A}'_{j,k}$ is an antichain for every $k$. 
    
    Finally, we show that the pruned trees
    \[
        \mathfrak{T}'(\mathfrak{v}) = \mathfrak{T}(\mathfrak{v}) \setminus \bigcup_k \mathfrak{A}_{j,k}
    \]
    are $2^{\gamma n}$-separated. Let $\mathfrak{p} \in \mathfrak{T}'(\mathfrak{v})$ and assume $I_\mathfrak{p} \subset I_{\mathfrak{v}'}$. Then there are tiles $\mathfrak{p}_1 < \dotsb < \mathfrak{p}_{\lceil\gamma n\rceil} < \mathfrak{p}$ in $\mathfrak{T}(\mathfrak{v})$. By the separation condition \eqref{SeparationCondition}, there exists some $Q \in \mathcal{Q}(\mathfrak{v}') \setminus \mathcal{Q}(10 \mathfrak{p}_1)$. It follows that
    \begin{align*}
        \|Q_{\mathfrak{p}} - Q_{\mathfrak{v}'}\|_{I_{\mathfrak{p}}} &\geq 10^{4\lceil\gamma n\rceil} \|Q_{\mathfrak{p}} - Q_{\mathfrak{v}'}\|_{I_{\mathfrak{p}_1}}\\
        &\geq 10^{4\lceil\gamma n\rceil} (-\|Q_{\mathfrak{p}} - Q_{\mathfrak{p}_1}\|_{I_{\mathfrak{p}_1}} + \|Q_{\mathfrak{p}_1} - Q\|_{I_{\mathfrak{p}_1}} - \|Q - Q_{\mathfrak{v}'}\|_{I_{\mathfrak{v}'}})\\
        &\geq 10^{4\lceil\gamma n\rceil}(-1 + 10 - 1) \geq 8 \cdot 10^{4\gamma n}\,,
    \end{align*}
    thus the trees $\mathfrak{T}'(\mathfrak{v})$ are $2^{\gamma n}$ separated.
    
    To summarize, we decomposed $\mathfrak{h}_n \setminus \mathfrak{C}_n$ as a union of at most $n$ antichains. Then we decomposed $\mathfrak{C}_n = \cup_{j = 0}^{C(n + \log \log \lambda)} \mathfrak{C}_{n,j}$. We showed that each $\mathfrak{C}_{n,j}$ is the union of the antichain $\mathfrak{A}_j$, $\lceil\gamma n \rceil$ antichains $\mathfrak{A}'_{j,k}$ and one $L^\infty$-forest $\cup_{\mathfrak{v} \in \mathfrak{V}} \mathfrak{T}'(\mathfrak{v})$.
\end{proof}

The sets $\mathfrak{h}_n$ form an increasing chain with union $\mathfrak{P}_{good}$, thus
\[
    \mathfrak{P}_{good} = \mathfrak{h}_{n_0} \cup \bigcup_{n_0}^\infty \mathfrak{h}_{n+1} \setminus \mathfrak{h}_n\,.
\]
Since the sets $\mathfrak{h}_{n}$ are down subsets, the differences $\mathfrak{h}_{n+1} \setminus \mathfrak{h}_{n}$ are convex, thus the intersection of any tree in $\mathfrak{h}_{n+1}$ with $\mathfrak{h}_{n+1}\setminus\mathfrak{h}_n$ is still a tree. Hence \cref{DecompositionTreeAntichain} also yields a decomposition of $\mathfrak{h}_{n+1} \setminus \mathfrak{h}_n$ into $O(\gamma n^2 + \gamma n \log\log \lambda)$ antichains and $O(n + \log\log \lambda)$ $L^\infty$-forests of level $n$. In summary, we have the following disjoint union (the notation differs from the proof of \cref{DecompositionTreeAntichain}):
\begin{align}
    \label{decomposition}
    \mathfrak{P}_{good} =  \bigcup_{n \geq n_0} \left( \bigcup_{j = 1}^{C(\gamma n^2 + \gamma n \log \log \lambda)} \mathfrak{A}_{n,j} \cup \bigcup_{j =1}^{C(n + \log\log \lambda)} \mathfrak{F}_{n,j}\right)
\end{align}
where 
\begin{itemize}
    \item each $\mathfrak{A}_{n,j}$ is an antichain,
    \item each $\mathfrak{F}_{n,j} = \cup_l \mathfrak{T}_{n,j,l}$ is an $L^\infty$-forest of level $n$,
    \item $\dens{\mathfrak{A}_{n,j}} \lesssim 2^{-n}$ if $n > n_0$,
    \item $\dens{\mathfrak{F}_{n,j}} \lesssim 2^{-n}$ if $n > n_0$.
\end{itemize}

\section{Error Terms}
In this section, we deal with the parts of the operators $T_\mathfrak{A}$ and $T_{\bd(\mathfrak{T})}$ corresponding to antichains and boundary parts of trees.

\subsection{Antichains}
We first prove estimates for the operators $T_\mathfrak{A}$ associated to antichains $\mathfrak{A}$. We briefly explain the strategy. For an antichain $\mathfrak{A}$, the sets $E(\mathfrak{p}), \mathfrak{p} \in \mathfrak{A}$ are disjoint. Therefore the pointwise estimate $|T_\mathfrak{A}f| \lesssim M^Kf$ holds. With this estimate, one can control one operator $T_\mathfrak{A}$, it is however not good enough to control the countable collection of operators $T_{\mathfrak{A}_{n,j}}$. Since there are $O(n^2 \log\log(\lambda)^2)$ antichains of each density parameter $n$, one can control all antichains by showing exponential decay of the operator norm of $T_{\mathfrak{A}_{n,j}}$ in $n$. In \cite{ZK2021}, this is done as follows: One derives a bound for $T_{\mathfrak{p}_1}T^*_{\mathfrak{p}_2}$ which decays like a power of the Fourier separation $\Delta(\mathfrak{p}_1, Q_{\mathfrak{p}_2})$. If this separation is of size $\gtrsim 2^{\varepsilon n}$, this bound is good enough. 
On the other hand, the measure of the union of all $E(\mathfrak{p}_2)$ with smaller separation is exponentially small since $\dens(\mathfrak{A}_{n,j}) \lesssim 2^{-n}$. 
This strategy carries over to the anisotropic setting without problems, and yields the estimate $\|T_{\mathfrak{A}}\|_{2 \to 2} \lesssim A 2^{-\varepsilon n}$. 
To obtain the improved constants in \cref{MainThmWeakBound1}, we will use this estimate for $n > n_0$ and the maximal function argument described above at density $n_0$, where we choose $n_0 \sim \log(A/(\|M^K\|_{2\to 2} + \|R^K\|_{2\to2})$. The square of the logarithm in \cref{MainThmWeakBound1} arises then because there are $O(n_0^2)$ antichains of density $n_0$.

We now carry out the details of this argument. We start with the maximal function estimate:

\begin{lemma}
    \label{TrivialAntichainBound}
    Let $\mathfrak{A}$ be an antichain. Then it holds that
    \[
        \|T_{\mathfrak{A}}\|_{2 \to 2} \lesssim \|M^K\|_{2\to 2}\,.
    \]
\end{lemma}

\begin{proof}
    Since $\mathfrak{A}$ is an antichain the sets $E(\mathfrak{p})$, $\mathfrak{p}\in \mathfrak{A}$ are pairwise disjoint. Therefore
    \begin{align*}
        |T_{\mathfrak{A}} f(x)| &= \left|\sum_{\mathfrak{p} \in  \mathfrak{A}} \mathbf{1}_{E(\mathfrak{p})}(x) \int K_{s(\mathfrak{p})}(x - y) e(Q_x(x) - Q_x(y)) f(y) \, \mathrm{d}y\right|\\
        &\leq \sum_{\mathfrak{p} \in \mathfrak{A}} \mathbf{1}_{E(\mathfrak{p})}(x)  \int |K_{s(\mathfrak{p})}(x - y)||f(y)| \, \mathrm{d}y\lesssim M^K f(x)\,. \qedhere
    \end{align*}
\end{proof}

To show estimates with decay in the density parameter $n$, we need the following van der Corput type estimate for oscillatory integrals with polynomial phase:
\begin{lemma}[\cite{ZK2021}, Lem. A1]
    \label{VanDerCorputA1}
    Let $\psi : \R^\mathbf{d} \to \C$ be a measurable function supported in $\delta_r(J)$ for an isotropic cube $J$. Then 
    \[
        \left|\int_{\R^\mathbf{d}} e^{iQ(x)}\psi(x) \, \mathrm{d}x\right| \lesssim \sup_{|y| < \Delta^{-1/d}l(J)} \int_{\R^\mathbf{d}} |\psi(x) - \psi(x - \delta_r(y))| \, \mathrm{d}x\,,
    \]
    where $l(J)$ is the side length of $J$ and $\Delta = \|Q\|_{\delta_r(J)} + 1$. 
\end{lemma}

\begin{proof}
    The lemma follows from  \cite{ZK2021}, Lemma A1, after precomposing everything with $\delta_r$.
\end{proof}

We have the following estimate for separated tiles:

\begin{lemma}[\cite{ZK2021}, Lem. 4.1]
    \label{BasicTT}
    Let $\mathfrak{p}_1, \mathfrak{p}_2 \in \mathfrak{P}$ with $|I_{\mathfrak{p}_1}| \leq |I_{\mathfrak{p}_2}|$. 
    Then 
    \begin{equation}
    \label{BasicTTequation}
        \left| \int  T_{\mathfrak{p}_1}^* g_1 \overline{T_{\mathfrak{p}_2}^* g_2}\,\mathrm{d}y\right| \lesssim A^2\, \frac{\Delta(\mathfrak{p}_1, Q_{\mathfrak{p}_2})^{-1/(\alpha_\mathbf{d} d)}}{|I_{\mathfrak{p}_2}|} \int_{E(\mathfrak{p}_1)} |g_1| \, \mathrm{d}x_1 \int_{E(\mathfrak{p}_2)} |g_2|\,\mathrm{d}x_2\,.
    \end{equation}
\end{lemma}

\begin{proof}
    Writing out the left hand side and pulling absolute values inside yields
    \begin{align}
    \label{TTexplanation}
        &\int_{E(\mathfrak{p}_1)} \int_{E(\mathfrak{p}_2)} g_1(x_1) g_2(x_2) \left|\int e((Q_{x_1} - Q_{x_2})(y)) \overline{K_{s(\mathfrak{p}_1)}(x_1 - y)}K_{s(\mathfrak{p}_2)}(x_2 -y) \, \mathrm{d}y \right| \mathrm{d}x_1 \mathrm{d}x_2 \,.
    \end{align}
    We want to estimate the inner integral using \cref{VanDerCorputA1}. By the assumptions \eqref{KernelUpperBound} and \eqref{KernelHoelderBound}, the product $\psi(y) = \overline{K_{s(\mathfrak{p}_1)}(x_1 - y)}K_{s(\mathfrak{p}_2)}(x_2 -y)$ satisfies
    \[
        |\psi(y) - \psi(y')| \lesssim A^2\, \frac{\rho(y - y')}{D^{s(\mathfrak{p}_1)(1 + |\alpha|)+ s(\mathfrak{p}_2)|\alpha|}}\,.
    \]
    Furthermore, since $x_1 \in I_{\mathfrak{p}_1}$ and by the support assumption \eqref{KernelSupport}, the function $\psi$ is supported in $I_{\mathfrak{p}_1}^*$. Applying \cref{VanDerCorputA1} with $J = \delta_{D^{-s(\mathfrak{p}_1)}}(I_{\mathfrak{p}_1}^*)$ and $r = D^{s(\mathfrak{p}_1)}$, we conclude that the inner integral in \eqref{TTexplanation} is bounded by
    \begin{align*}
        \sup_{|y| < 3\Delta^{-1/d}} \int |\psi(x) - \psi(x - \delta_{D^{s(\mathfrak{p}_1)}}(y))| \, \mathrm{d}x&\lesssim D^{|\alpha|s(\mathfrak{p}_1)}\, A^2 \, D^{s(\mathfrak{p}_1)}\,\frac{ \sup_{|y| < 3\Delta^{-1/d}} \rho(y) }{D^{s(\mathfrak{p}_1)(1 + |\alpha|)+ s(\mathfrak{p}_2)|\alpha|}}\\
        &\lesssim A^2\, \frac{\Delta^{-1/(\alpha_{\mathbf{d}} d)}}{|I_{\mathfrak{p}_2}|}\,,
    \end{align*}
    where $\Delta = \|Q_{x_1} - Q_{x_2}\|_{I_{\mathfrak{p}_1}^*} + 1\geq \|Q_{x_1} - Q_{x_2}\|_{I_{\mathfrak{p}_1}} +1$. The left hand side of the inequality \eqref{BasicTTequation} is only nonzero when $I_{\mathfrak{p}_1}^*$ and $I_{\mathfrak{p}_2}^*$ intersect.
    In that case, since $|I_{\mathfrak{p}_1}| \leq |I_{\mathfrak{p}_2}|$, it holds that $I_{\mathfrak{p}_1} \subset 5I_{\mathfrak{p}_2}$. Therefore $\|\cdot\|_{I_{\mathfrak{p}_1}} \leq \|\cdot\|_{5I_{\mathfrak{p}_2}} \leq C \|\cdot\|_{I_{\mathfrak{p}_2}}$, by \cref{PolynomialBound}. Using this, we estimate
    \begin{align*}
        \|Q_{x_1} - Q_{x_2}\|_{I_{\mathfrak{p}_1}} 
        &\geq -\|Q_{x_1} - Q_{\mathfrak{p}_1}\|_{I_{\mathfrak{p}_1}}
        + \|Q_{\mathfrak{p}_1} - Q_{\mathfrak{p}_2}\|_{I_{\mathfrak{p}_1}}
        - \|Q_{\mathfrak{p}_2} - Q_{x_2}\|_{I_{\mathfrak{p}_1}}\\
        &\geq -1 + (\Delta(\mathfrak{p}_1, Q_{\mathfrak{p}_2}) - 1)  - C\,.
    \end{align*}
    Since it trivially holds that $\Delta \geq 1$ this implies that $\Delta \gtrsim \Delta(\mathfrak{p}_1, Q_{\mathfrak{p}_2})$,
    which concludes the proof.
\end{proof}

The following lemma will allow us to exploit the small density of the antichains $\mathfrak{A}_{n,j}$.

\begin{lemma}[\cite{ZK2021}, Lem. 4.3]
    \label{AntichainsSeperation}
    There exists $\varepsilon > 0$ such that if $0 \leq \eta \leq 1$,  $1 \leq p \leq \infty$, $Q \in \mathcal{Q}$ and $\mathfrak{A} \subset \mathfrak{P}$ is an antichain, then it holds that
    \[
        \left\| \sum_{\mathfrak{p} \in \mathfrak{A}} \Delta(\mathfrak{p}, Q)^{-\eta} \mathbf{1}_{E(\mathfrak{p})}\right\|_p \lesssim \dens(\mathfrak{A})^{\varepsilon \eta/p} \left|\bigcup_{\mathfrak{p} \in \mathfrak{A}} I_{\mathfrak{p}}\right|^{1/p}\,.
    \]
\end{lemma}

\begin{proof}
    Note that the statement holds for $p = \infty$ since the sets $E(\mathfrak{p})$ are disjoint and $\Delta(\mathfrak{p}, Q) \geq 1$. Therefore, it suffices to show it for $p =1$, then the case of general $p$ follows from Hölder's inequality. Thus we have to show that
    \begin{align*}
        \sum_{\mathfrak{p} \in \mathfrak{A}} |E(\mathfrak{p})| \Delta(\mathfrak{p},Q)^{-\eta} \lesssim \dens(\mathfrak{A})^{\varepsilon \eta} \left|\bigcup I_{\mathfrak{p}}\right|\,.
    \end{align*}
    Fix $\varepsilon>0$ to be chosen later. Since the sets $E(\mathfrak{p})$ are disjoint and contained in $\bigcup I_\mathfrak{p}$, the estimate holds if one restricts to tiles $\mathfrak{p}$ with $\Delta(\mathfrak{p},Q) \geq \dens(\mathfrak{A})^{-\varepsilon}$. Denote the collection of all other tiles by $\mathfrak{A}' = \{\mathfrak{p} \in \mathfrak{A} \, : \, \Delta(\mathfrak{p},Q) < \dens(\mathfrak{A})^{-\varepsilon}\}$. Let $\mathcal{L}$ be the set of maximal $D$-adic cubes $L$ such that $L \subsetneq I_\mathfrak{p}$ for some $\mathfrak{p} \in \mathfrak{A}'$ but $I_\mathfrak{p} \not \subset L$ for all $\mathfrak{p} \in \mathfrak{A}'$. The cubes $L$ form a disjoint cover of $\cup_{\mathfrak{p} \in \mathfrak{A}'} I_\mathfrak{p}$. Therefore, it suffices to show that
    \[
        |E(L)| \lesssim \dens(\mathfrak{A})^{1 - \varepsilon \dim \mathcal{Q}} |L| 
    \]
    for all $L \in \mathcal{L}$, where $E(L) = \cup_{\mathfrak{p} \in \mathfrak{A}'} E(\mathfrak{p}) \cap L$. Fix $L \in \mathcal{L}$. There exists some tile $\mathfrak{p} \in \mathfrak{A}'$ such that $I_\mathfrak{p} \subset \hat L$. Let $\mathfrak{p}' = \mathfrak{p}$ if $I_\mathfrak{p} = \hat L$ and else let $\mathfrak{p}'$ be the unique tile with $Q \in \mathcal{Q}( \mathfrak{p}')$ and $I_{\mathfrak{p}'} = \hat{L}$. We claim that for $a = 3\dens(\mathfrak{A})^{-\varepsilon}$, the tile $\mathfrak{p}'$ satisfies:
    \begin{itemize}
        \item $a\mathfrak{p} \leq a \mathfrak{p}'$
        \item for every $\mathfrak{p}'' \in \mathfrak{A}'$ with $L\cap I_{\mathfrak{p}''} \neq \emptyset$ we have $a \mathfrak{p}' \leq \mathfrak{p}''$.
    \end{itemize}
    Indeed, the first point clearly holds for $\mathfrak{p}' =  \mathfrak{p}$. In the other case, we use that for every $\tilde Q \in \mathcal{Q}(a\mathfrak{p}')$ we have that
    \begin{align*}
        \|\tilde Q - Q_\mathfrak{p}\|_{I_\mathfrak{p}} &\leq \|\tilde Q - Q_{\mathfrak{p}'}\|_{I_\mathfrak{p}} + \|Q_{\mathfrak{p}'} - Q\|_{I_\mathfrak{p}} + \|Q - Q_\mathfrak{p}\|_{I_\mathfrak{p}} \\
        &\leq 10^{-4} (\|\tilde Q - Q_{\mathfrak{p}'}\|_{I_{\mathfrak{p}'}} + \|Q_{\mathfrak{p}'} - Q\|_{I_{\mathfrak{p}'}}) + \|Q - Q_\mathfrak{p}\|_{I_\mathfrak{p}}\\
        &\leq 10^{-4} (3\dens{\mathfrak{A}}^{-\varepsilon} + 1) + \dens(\mathfrak{A})^{-\varepsilon} \leq 3\dens{\mathfrak{A}}^{-\varepsilon}\,.
    \end{align*}
    For the second point we note that the inclusion $I_{\mathfrak{p}'} = \hat L \subset I_{\mathfrak{p}''}$ holds by definition of $\mathcal{L}$. For the other inclusion let $\tilde Q \in \mathcal{Q}(\mathfrak{p}'')$. Then 
    \begin{align*}
        \|\tilde Q - Q_{\mathfrak{p}'}\|_{I_{\mathfrak{p}'}} &\leq
        \|\tilde Q - Q_{\mathfrak{p}''}\|_{I_{\mathfrak{p}'}}
        +\|Q_{\mathfrak{p}''} - Q\|_{I_{\mathfrak{p}'}}
        +\|Q - Q_{\mathfrak{p}'}\|_{I_{\mathfrak{p}'}}\\
        &\leq \|\tilde Q - Q_{\mathfrak{p}''}\|_{I_{\mathfrak{p}''}}
        +\|Q_{\mathfrak{p}''} - Q\|_{I_{\mathfrak{p}'}}
        +\|Q - Q_{\mathfrak{p}'}\|_{I_{\mathfrak{p}'}}\\
        &\leq 1 + \dens(\mathfrak{A})^{-\varepsilon} + \dens(\mathfrak{A})^{-\varepsilon} \leq 3\dens(\mathfrak{A})^{-\varepsilon}\,.
    \end{align*} 
    The second estimate holds since $\mathfrak{p}'' \in \mathfrak{A}'$, and since $Q \in \mathcal{Q}(\mathfrak{p}')$
    or $\mathfrak{p}' = \mathfrak{p} \in \mathfrak{A}'$, which gives $\|Q - Q_{\mathfrak{p}'}\|_{I_{\mathfrak{p}'}} \leq \max\{1, \dens(\mathfrak{A})^{-\varepsilon}\} = \dens(\mathfrak{A})^{-\varepsilon}$.
    We conclude that the claim holds.
    
    The second point of the claim implies that $E(L) \subset \overline{E}(a \mathfrak{p}')$. Thus, by the definition \eqref{densedefinition} of density,
    \[
        |E(L)| \leq | \overline{E}(a \mathfrak{p}')| \leq \dens(\mathfrak{A}) a^{\dim \mathcal{Q}} |I_{\mathfrak{p}'}| \lesssim \dens(\mathfrak{A})^{1 - \varepsilon\dim\mathcal{Q}} |L|\,.
    \]
    The lemma now follows with $\varepsilon = 1/(1 + \dim \mathcal{Q})$.
\end{proof}

We finally combine the previous two lemmas to show an estimate with exponential decay in $n$.

\begin{lemma}[\cite{ZK2021}, Prop. 4.6]
    \label{AntichainBound}
    Fix $n > n_0$ and fix $j$. Let $\mathfrak{A} = \mathfrak{A}_{n,j}$. Then
    \[
        \|T_{\mathfrak{A}}\|_{2 \to 2} \lesssim A 2^{-n\varepsilon} \,.
    \]
\end{lemma}

\begin{proof}
    Denote for each tile $\mathfrak{p} \in \mathfrak{A}$ 
    \[
        \mathfrak{D}(\mathfrak{p}) = \{\mathfrak{p}' \in \mathfrak{A} \, : \, I_{\mathfrak{p}}^* \cap I_{\mathfrak{p}'}^* \neq \emptyset\,, \, s(\mathfrak{p}') \leq s(\mathfrak{p})\}\,.
    \]
    The operator $T_{\mathfrak{p}}T_{\mathfrak{p}'}^*$ is zero unless $\mathfrak{p} \in \mathfrak{D}(\mathfrak{p}')$ or $\mathfrak{p}' \in \mathfrak{D}(\mathfrak{p})$. Furthermore, $\mathfrak{p}' \in \mathfrak{D}(\mathfrak{p})$ implies $I_{\mathfrak{p}'} \subset 5I_{\mathfrak{p}}$. Thus, by \cref{BasicTT} and Hölder's inequality, it holds with $1 < q < 2$:
    \begin{align*}
        \|T_{\mathfrak{A}}^*g\|_2^2 
        &\leq 2 \sum_{\mathfrak{p} \in \mathfrak{A}} \sum_{\mathfrak{p}' \in \mathfrak{D}(\mathfrak{p})} \left|\int T_{\mathfrak{p}}^* g \overline{T_{\mathfrak{p}'}^* g } \right|\\
        &\lesssim A^2 \,\sum_{\mathfrak{p} \in \mathfrak{A}} \int_{E(\mathfrak{p})} |g|  \frac{1}{|I_{\mathfrak{p}}|}\sum_{\mathfrak{p}' \in \mathfrak{D}(\mathfrak{p})} \Delta(\mathfrak{p}', Q_{\mathfrak{p}})^{-1/(\alpha_nd)} \int_{E(\mathfrak{p}')} |g| \\
        &\leq A^2 \,\sum_{\mathfrak{p} \in \mathfrak{A}} \int_{E(\mathfrak{p})} |g| \left(\frac{1}{|I_\mathfrak{p}|}\int_{5I_\mathfrak{p}} |g|^q\right)^{1/q} \frac{\left\|\sum_{\mathfrak{p}' \in \mathfrak{D}(\mathfrak{p})} \mathbf{1}_{E(\mathfrak{p}')} \Delta(\mathfrak{p}', Q_\mathfrak{p})^{-1/(\alpha_n d)}\right\|_{q'}}{|I_\mathfrak{p}|^{1/q'}}\,.
    \end{align*}
    By \cref{AntichainsSeperation}, the last fraction is $\lesssim 2^{-\varepsilon n}$. Hence we can finish the chain of estimates with
    \begin{align*}
        \lesssim A^2\, 2^{-\varepsilon n}\int |g| M^q |g| 
        \lesssim A^2 \,2^{-\varepsilon n} \|g\|_2^2\,,
    \end{align*}
    using $L^2$ boundedness of the $q$-maximal function for $q < 2$. 
\end{proof}

\subsection{Boundary Parts of Trees}
\label{BoundaryTreesSection}
Let $\mathfrak{T}$ be a tree. Define its boundary part 
\[
    \bd(\mathfrak{T}) = \{\mathfrak{p} \in \mathfrak{T} \, : \, I_{\mathfrak{p}}^* \not\subset I_\mathfrak{T}\}\,.
\]
We deal with the boundary parts of trees separately because we will later need that $T_\mathfrak{T}^* f$ is supported in $I_\mathfrak{T}$, which only holds after removing the boundary part from $\mathfrak{T}$. 

Our treatment of the boundary parts follows the argument in \cite{Fefferman1973} and \cite{Lie2009}, and differs from the arguments in \cite{Lie2020}, \cite{ZK2021}.
The strategy is as follows: The boundary tiles in a tree that are $k$ scales below the top form an antichain, hence we can estimate them using the results of the previous section. Furthermore, the total spatial support of these tiles decays exponentially in $k$, so we need to estimate only $\sim n \log(\lambda)$ layers and can pack the rest into an exceptional set. 

We have the following estimates for the operator associated to the set of all boundary tiles in a forests $\mathfrak{F}_{n,j}$:
\begin{lemma}
    \label{BoundaryTrees}
    Let $n \geq 3$ and $j$ be given. Consider the collection of tiles $\mathfrak{S} = \bigcup_l \bd(\mathfrak{T}_{n,j,l})$. There exists an exceptional set $E_2 = E_2(n,j)$ with $|E_2| \lesssim 2^{-n}\lambda^{-2}$ such that 
    \[
        \|\mathbf{1}_{\mathbb{T}^\mathbf{d} \setminus E_2} T_{\mathfrak{S}}\|_{2 \to 2} \lesssim 
        \begin{cases}
            n \log(\lambda)  A2^{-n\varepsilon} & \text{if $n > n_0$}\\
            n \log(\lambda)  \|M^K\|_{2 \to 2} & \text{if $n = n_0$}
        \end{cases}\,.
    \]
\end{lemma}

\begin{proof}
    For $k\geq 1$ and a tree $\mathfrak{T}$, let $\bd_k(\mathfrak{T})$ be the set of tiles $\mathfrak{p} \in \bd(\mathfrak{T})$ with $s(\mathfrak{p}) = s(\tp \mathfrak{T}) - k$. The tiles in $\bd_k(\mathfrak{T}_{n,j,l})$ form an antichain as they all have the same scale. 
    Tiles $\mathfrak{p} \in \mathfrak{T},\mathfrak{p}' \in \mathfrak{T}'$ in separated trees are also not comparable.
    Indeed, else we had without loss of generality that $I_\mathfrak{p} \subset I_{\mathfrak{p}'}$.
    Then, by separation and the fact that $\gamma n \geq n \geq 3$, we have $\|Q_{\mathfrak{p}} - Q_{\mathfrak{T}'}\|_{I_\mathfrak{p}} \geq 2^{\gamma n} - 1 > 5$. 
    Since $Q_{\mathfrak{T}'} \in \mathcal{Q}(4\mathfrak{p}')$, it follows that $\mathcal{Q}(\mathfrak{p}') \not\subset \mathcal{Q}(\mathfrak{p})$.
    We conclude that the set $\mathfrak{S}_k = \cup_l \bd_k(\mathfrak{T}_{n,j,l})$ is an antichain. 
    
    Fix $c = 10n \log(\lambda)$. By \cref{TrivialAntichainBound} and \cref{AntichainBound}, it holds that 
    \[
        \|\sum_{k \leq c} T_{\mathfrak{S}_k}\|_{2 \to 2} \lesssim 
        \begin{cases}
            n \log(\lambda)  A2^{-n\varepsilon} & \text{if $n > n_0$}\\
            n \log(\lambda)  \|M^K\|_{2 \to 2} & \text{if $n = n_0$}
        \end{cases}\,.
    \]
    We show that for each $f \in L^2(\mathbb{T}^\mathbf{d})$ the remainder $T_{\mathfrak{S} \setminus \cup_{k \leq c} \mathfrak{S}_k} f$ has support in a small set $E_2$. Define 
    \[
        E_2 = \bigcup_{\mathfrak{p} \in \cup_{k > c} \mathfrak{S}_k} I_\mathfrak{p} 
        = \bigcup_{l} \bigcup_{\mathfrak{p} \in \bd(\mathfrak{T}_{n,j,l})\,:\, s(\mathfrak{p}) < s(\tp \mathfrak{T}_{n,j,l})- c} I_\mathfrak{p} \,,
    \]
    so that $T_{\cup_{k > c} \mathfrak{S}_k}f$ is supported in $E_2$ for all $f$.
    For fixed $l$, the measure of the inner union is $\lesssim D^{-c\alpha_1} |I_{ \mathfrak{T}_{n,j,l}}|$.
    Thus 
    \[
        |E_2| \lesssim \sum_{l} D^{-c\alpha_1} |I_{\mathfrak{T}_{n,j,l}}| \lesssim 2^n \log(n+1) \log(\lambda) D^{-10\alpha_1 n \log(\lambda)} \leq 2^{-n}\lambda^{-2}\,,
    \]
    where we used that $\sum_l |I_{\mathfrak{T}_{n,j,l}}| \lesssim 2^n \log(n+1) \log(\lambda)$ by \eqref{ForestEqn} and that $D\geq 2$. This concludes the proof.
\end{proof}

\section{Trees and Forests}
In this section, we prove estimates for operators $T_\mathfrak{T}$ and $T_\mathfrak{F}$ corresponding to trees $\mathfrak{T}$ or forests $\mathfrak{F}$. 
The basic idea is that, given a tree $\mathfrak{T}$ and $x \in E(\mathfrak{p})$ for some tile $\mathfrak{p} \in \mathfrak{T}$, the phase $Q_x$ is \enquote{close} to $Q_\mathfrak{T}$. 
If one replaces $Q_x$ by $Q_\mathfrak{T}$, the resulting operator is bounded by $R^K(M_{-Q_\mathfrak{T}}f)$. On the other hand, the definition of a tree is chosen exactly so that the error in this replacement is bounded by the maximal average $M^K$. Thus the operator associated to a single tree is bounded on $L^2$, with norm bounded by $\|M^K\|_{2 \to 2} + \|R^K\|_{2 \to 2}$.  Using boundedness of the nontangential version of the maximally truncated singular integral, we then show estimates for single trees with exponential decay in the density parameter $n$ when $n > n_0$, however the constant in these estimates is of order $A$.

Next, we bound the operators associated to forests. Using an orthogonality argument, we can control collections of trees with spatially disjoint tops, so called rows. After that we show an estimate for $T_{\mathfrak{T}_1}T_{\mathfrak{T}_2}^*$ for separated trees $\mathfrak{T}_1$, $\mathfrak{T}_2$ with power decay in the separation, and use it to show a similar bound for $T_{\mathfrak{R}_1}T_{\mathfrak{R}_2}^*$ for separated rows.
With these two ingredients, we can then control a whole $L^\infty$-forest by splitting it into rows and using that these rows are $2^{\gamma n}$-separated.
In this part of the proof the argument for separated rows, \cref{SepRowBound}, has constants proportional to $A$. This is problematic in the case $n = n_0$, since there we want to obtain an estimate by $\|M^K\|_{2 \to 2} + \|R^K\|_{2 \to 2}$. However, $n_0$ will be chosen sufficiently large so that the $2^{\gamma n_0}$ separation of the trees cancels the constant $A$.

\subsection{Basic Estimates for Trees}

As described above, the operators $T_\mathfrak{T}$ are bounded on $L^2(\mathbb{T}^\mathbf{d})$ with norms depending only on the norms of $M^K$ and $R^K$:
\begin{lemma}
    \label{TrivialTreeBound}
    Let $\mathfrak{T} \subset \mathfrak{P}$ be a tree and assume that $f,g \in L^2(\mathbb{T}^\mathbf{d})$. Then it holds that
    \[
        \left| \int_{\mathbb{T}^\mathbf{d}} g T_\mathfrak{T} f \right| \lesssim  (\|M^K\|_{2 \to 2} + \|R^K\|_{2\to2})\| f\|_2 \| g\|_{2}\,.
    \]
\end{lemma}

\begin{proof}
    We fix $x$. Define 
    \[
        \sigma = \sigma(\mathfrak{T}, x) = \{s \, : \, \exists \mathfrak{p} \in \mathfrak{T}, \, s(\mathfrak{p}) = s, \, x \in E(\mathfrak{p})\}\,.
    \]
    This is the set of all scales $s$ for which there exists a tile contributing to $T_\mathfrak{T}f(x)$. It is a convex subset of $\mathbb{Z}$ since $\mathfrak{T}$ is convex and if $\mathfrak{p} \leq \mathfrak{p}' \leq \mathfrak{p}''$ then $E(\mathfrak{p}) \cap E(\mathfrak{p}'') \subset E(\mathfrak{p}')$.
    
    We write
    \begin{align*}
        |T_\mathfrak{T} f(x)| &= \left| \sum_{s \in \sigma} \int e(-Q_x(y)) K_s(x-y) f(y) \, \mathrm{d}y \right|\\
        &\leq \left| \sum_{s \in \sigma} \int e(-Q_\mathfrak{T}(y)+Q_\mathfrak{T}(x)) K_s(x-y) f(y) \, \mathrm{d}y \right| \\
        &\quad + \sum_{s \in \sigma} \int |e(-Q_\mathfrak{T}(y) + Q_\mathfrak{T}(x)+Q_x(y) - Q_x(x)) - 1| |K_s(x-y)||f(y)| \, \mathrm{d}y\\
        &\eqqcolon A(x) + B(x)\,.
    \end{align*}
    The first term $A(x)$ is bounded by $R^{K}(M_{-Q_\mathfrak{T}} f)(x)$. To treat $B(x)$, note that if $K_s(x-y) \neq 0$, then $\rho(x-y) \leq D^s/4$. Therefore, for all such $x,y$, it holds that
    \begin{align*}
        |e(-Q_\mathfrak{T}(y) + Q_\mathfrak{T}(x)+Q_x(y) - Q_x(x)) - 1| &\lesssim \|Q_x - Q_\mathfrak{T}\|_{B_\rho(x, D^s/4)}\\
        &\lesssim D^{\alpha_1(s - \max \sigma)} \|Q_x - Q_\mathfrak{T}\|_{B_\rho(x, D^{\max \sigma}/4)} \,,
    \end{align*}
    where we applied \cref{PolynomialBound}. There exists a tile $\mathfrak{p} \in \mathfrak{T}$ of scale $\max \sigma$ with $x \in E(\mathfrak{p})$. Using \cref{PolynomialBound} once more, we can estimate
    \[
         \|Q_x - Q_\mathfrak{T}\|_{B_\rho(x, D^{\max \sigma}/4)} \leq \|Q_x - Q_\mathfrak{T}\|_{I_\mathfrak{p}^*} \lesssim \|Q_x - Q_\mathfrak{T}\|_{I_\mathfrak{p}} \leq 5\,.
    \]
    Therefore we have
    \begin{align*}
        B(x)&\leq \sum_{s \in \sigma} D^{\alpha_1(s-\max \sigma)}  \int  |K_s(x-y)||f(y)| \, \mathrm{d}y\\
        &\lesssim \sum_{s \leq \max \sigma} D^{\alpha_1(s - \max \sigma)} M^Kf(x)\lesssim M^Kf(x)\,.
    \end{align*}
    This completes the proof.
\end{proof}

Now we want to show estimates with decay in $\dens(\mathfrak{T})$. We will need the following well known result:
\begin{lemma}
\label{nontangential}
The nontangential maximal function 
\[
    R_\mathcal{N}^Kf(x) = \sup_{\underline{\sigma} \leq \overline{\sigma}} \sup_{\rho(x-z) \leq C D^{\underline{\sigma}}} \left|\sum_{s = \underline{\sigma}}^{\overline{\sigma}} \int K_s(z-y) f(y) \, \mathrm{d}y \right|
\]
satisfies the estimate
\[
    \|R_\mathcal{N}^K f\|_{2 \to 2} \lesssim_{C} A\,.
\]
\end{lemma}
\begin{proof}
    See \cite{Stein1993}, I. 7.3. Note that the situation there is slightly different. In particular, the bound is shown with the inner supremum over $\rho(x-y) \leq cD^{\underline{\sigma}}$, for some small $c$, and with a sharp truncation of the singular integral. It is however possible to change the constant $c$ while losing only a constant factor, see \cite{Stein1993} II. 2.5.1. Furthermore the difference between the sharply truncated singular integral and the smoothly truncated singular integral is bounded by $A M$, where $M$ is the Hardy-Littlewood maximal function.
\end{proof}

We further need two definitions from \cite{ZK2021}.

\begin{definition}
    For a finite collection of tiles $\mathfrak{S} \subset  \mathfrak{P}$ denote
    \begin{itemize}
        \item by $\mathcal{J}(\mathfrak{S})$ the collection of maximal $D$-adic cubes $J$ such that $100 D J$ contains no $I_\mathfrak{p}$, $\mathfrak{p} \in \mathfrak{S}$. 
        \item by $\mathcal{L}(\mathfrak{S})$ the collection of maximal $D$-adic cubes $L$ such that $L \subset I_\mathfrak{p}$ for some $\mathfrak{p} \in \mathfrak{S}$ and $I_\mathfrak{p} \not\subset L$ for all $\mathfrak{p} \in \mathfrak{S}$. 
    \end{itemize}
    For a set of pairwise disjoint $D$-adic cubes $\mathcal{J}$ we define the projection
    \[
        P_\mathcal{J} f = \sum_{J \in \mathcal{J}} \mathbf{1}_J \frac{1}{|J|} \int_J f\,.
    \]
\end{definition}

The following is a refined version of \cref{TrivialTreeBound}. 
\begin{lemma}[\cite{ZK2021}, Thm. 5.6]
    \label{ProjTreeBound}
    Let $\mathfrak{T} \subset \mathfrak{P}$ be a tree and denote $\mathcal{L} = \mathcal{L}(\mathfrak{T})$ and $\mathcal{J} = \mathcal{J}(\mathfrak{T})$. Suppose that $f, g \in L^2(\mathbb{T}^\mathbf{d})$. Then it holds that
    \[
        \left| \int_{\mathbb{T}^\mathbf{d}} g T_\mathfrak{T} f \right| \lesssim A \|P_{\mathcal{J}} |f|\|_2 \|P_\mathcal{L} |g|\|_{2}\,.
    \]
\end{lemma}

\begin{proof}
    Fix $L\in \mathcal{L}$. We will show that for $x \in L$
    \[
        |T_\mathfrak{T} f(x)| \lesssim \inf_{z \in L} R^K_\mathcal{N} P_\mathcal{J}(M_{-Q_\mathfrak{T}} f) (z) + A( \inf_{z \in L} M P_\mathcal{J} |f|(z) + \inf_{z \in L} S P_\mathcal{J}|f|(z))\,,
    \]
    where $M$ is the Hardy-Littlewood maximal function and $S$ is an operator depending on $\mathfrak{T}$ with $\|S\|_{2 \to 2} \lesssim 1$. This implies the claimed inequality by \cref{nontangential} and the $L^2$ boundedness of the Hardy-Littlewood maximal function.
    
    Define $\sigma = \sigma(\mathfrak{T},x)$ as in the proof of \cref{TrivialTreeBound} and write
    \begin{align*}
        |T_\mathfrak{T} f(x)| &= \left| \sum_{s \in \sigma} \int e(-Q_x(y)+Q_x(x)) K_s(x-y)f(y) \, \mathrm{d}y \right|\\
        &\leq \left| \sum_{s \in \sigma} \int  K_s(x-y)e(Q_\mathfrak{T}(x)) P_\mathcal{J}(e(-Q_\mathfrak{T})f)(y) \, \mathrm{d}y \right|\\
        &+ \left| \sum_{s \in \sigma} \int  K_s(x-y)e(Q_{\mathfrak{T}}(x))(1 - P_\mathcal{J})(e(-Q_\mathfrak{T})f)(y) \, \mathrm{d}y \right|\\
        &+ \sum_{s \in \sigma} \int |e(-Q_\mathfrak{T}(y) + Q_\mathfrak{T}(x)+Q_x(y) - Q_x(x)) - 1||K_s(x-y)||f(y)| \, \mathrm{d}y\\
        &\eqqcolon A(x) + B(x) + C(x)\,.
    \end{align*}
    Consider the first term $A(x)$. Since $x \in I_\mathfrak{p}$ for some tile $\mathfrak{p} \in \mathfrak{T}$ with $s(\mathfrak{p}) = \min \sigma$, it holds by the definition of $\mathcal{L}$ that $s(L) < \min \sigma$. Thus $\rho(x-z) \lesssim D^{ \min \sigma}$ for all $z \in L$, which implies that 
    \[
        A(x) \leq \inf_{z \in L} R^K_\mathcal{N} P_\mathcal{J}(M_{-Q_\mathfrak{T}} f)(z)\,.
    \]
    We turn to $C(x)$.
    In the proof of \cref{TrivialTreeBound}, we showed that for all $y$ with $K_s(x-y) \neq 0$, it holds that
    \begin{align*}
        |e(-Q_\mathfrak{T}(y) + Q_\mathfrak{T}(x)+Q_x(y) - Q_x(x)) - 1| \lesssim D^{\alpha_1(s -\max\sigma)}\,.
    \end{align*}
    Combining this, the upper bound \eqref{KernelUpperBound} for $K_s$, and the fact that $\mathcal{J}$ is a partition of $\mathbb{T}^\mathbf{d}$ we obtain 
    \begin{align*}
        C(x) &\lesssim A \sum_{s \in \sigma} D^{\alpha_1(s - \max \sigma)} D^{-s|\alpha|} \sum_{J \in \mathcal{J} \, : \, J \cap B_\rho(x, D^s/4) \neq \emptyset} \int_J |f(y)| \, \mathrm{d}y\,.
    \end{align*}
    Note that this expression does not change upon replacing $|f|$ with  $P_\mathcal{J}|f|$.
    If $J$ intersects $B_\rho(x, D^s/4)$ and $\mathfrak{p} \in \mathfrak{T}$ is a tile of scale $s$ with $x \in E(\mathfrak{p})$, then $J$ also intersects $I_\mathfrak{p}^*$. If $s(J) \geq s(\mathfrak{p})$ then this implies $I_\mathfrak{p} \subset 3J$, which contradicts the definition of $\mathcal{J}$. Thus $J \subset I_\mathfrak{p}^* \subset B_\rho(x, CD^s)$. As observed above, it holds that $\rho(x-z) \lesssim D^s$ for all $z \in L$. We conclude that, for all $z \in L$ and all $J \in \mathcal{J}$ with $J \cap B_\rho(x, D^s/4) \neq \emptyset$, it holds that $J \subset B_\rho(z, CD^s)$. Hence 
    \begin{align*}
        C(x) \lesssim A \sum_{s \in \sigma} D^{\alpha_1(s - \max \sigma)} \inf_{z \in L} M P_\mathcal{J} |f|(z) \lesssim A \inf_{z \in L} M P_\mathcal{J} |f|(z)\,.
    \end{align*}
    It remains to take care of the second term $B$. Denoting $h(y) = e(-Q_\mathfrak{T}(y))f(y)$, we have that
    \begin{align*}
        B(x)&=\sum_{\mathfrak{p} \in \mathfrak{T}} \mathbf{1}_{E(\mathfrak{p})} \sum_{J \in \mathcal{J}, J \subset 3I_\mathfrak{p}} \int_J K_{s(\mathfrak{p})}(x-y)(h(y) - \frac{1}{|J|}\int_J h(z) \, \mathrm{d}z) \, \mathrm{d}y\\
        &= \sum_{\mathfrak{p} \in \mathfrak{T}} \mathbf{1}_{E(\mathfrak{p})} \sum_{J \in \mathcal{J}, J \subset 3I_\mathfrak{p}} \int_J (K_{s(\mathfrak{p})}(x-y) - \frac{1}{|J|}\int_J K_{s(\mathfrak{p})}(x-z) \, \mathrm{d}z)h(y) \, \mathrm{d}y\\
        &\leq A \sum_{\mathfrak{p} \in \mathfrak{T}} \mathbf{1}_{E(\mathfrak{p})} \sum_{J \in \mathcal{J}, J \subset 3I_\mathfrak{p}} \frac{\diam_\rho(J)}{D^{s(\mathfrak{p})|\alpha| +1}} \int_J |f(y)| \, \mathrm{d}y\\
        &\leq A \sum_{I \in \mathcal{H}} \mathbf{1}_I \sum_{J \in \mathcal{J}, J \subset 3I} \frac{\diam_\rho(J)}{D^{s(I)|\alpha| + 1}} \int_J P_\mathcal{J}|f|\,\mathrm{d}y\,,
    \end{align*}
    where $\mathcal{H} = \{I_\mathfrak{p} \, : \, \mathfrak{p} \in \mathfrak{T}\}$. Define $S$ by
    $$
        S f(x) = \sum_{I \in \mathcal{H}} \mathbf{1}_I \sum_{J \in \mathcal{J}, J \subset 3I} D^{s(J) - s(I)} \frac{1}{|I|} \int_J f\,\mathrm{d}y\,.
    $$
    Using that $\diam_\rho(J) \lesssim D^{s(J)}$, we can then estimate 
    $$
        B(x) \lesssim  A S P_\mathcal{J}|f|(x)\,.
    $$
    Note that $SP_\mathcal{J}|f|$ is constant on all cubes $L \in \mathcal{L}$, and therefore we have that
    $SP_\mathcal{J}|f|(x) = \inf_{z \in L} SP_\mathcal{J}|f|(z)$. Finally, we show that $\|S\|_{2 \to 2} \lesssim 1$. It holds that
    \begin{align*}
        \left|\int g Sf  \, \mathrm{d}x \right|&\lesssim \sum_{J \in \mathcal{J}} \int_J |f(y)| \, \mathrm{d}y \sum_{I \in \mathcal{H}, J \subset 3I} D^{s(J) - s(I)} \frac{1}{|I|}\int_I |g(x)| \, \mathrm{d}x \\
        &\lesssim \sum_{J \in \mathcal{J}} \int_J |f(y)| Mg(y) \, \mathrm{d}y \sum_{I \in \mathcal{H}, J \subset 3I} D^{s(J) - s(I)}\\
        &\lesssim \int |f(y)| Mg(y) \, \mathrm{d}y \lesssim \|f\|_2 \|g\|_2\,.
    \end{align*}
    This completes the proof.
\end{proof}

Finally, we use \cref{ProjTreeBound} to obtain an estimate for trees with decay in the density $\dens(\mathfrak{T})$.
\begin{lemma}[\cite{ZK2021}, Cor. 5.10]
    \label{SumTreeBound}
    Let $\mathfrak{T} \subset \mathfrak{P}$ be a tree. Then it holds that
    \[
        \|T_\mathfrak{T}\|_{2\to2} \lesssim  A \dens(\mathfrak{T})^{1/2}\,.
    \]
\end{lemma}

\begin{proof}
    Let $L \in \mathcal{L}(\mathfrak{T})$. 
    By definition of $\mathcal{L}$, there exists some tile $\mathfrak{p}_L \in \mathfrak{T}$ with $I_{\mathfrak{p}_L} \subset \hat L$. 
    Let $\mathfrak{p}'$ be the unique tile with $\mathfrak{p}_L \leq \mathfrak{p}'$, $I_{\mathfrak{p}'} = \hat L$ and $Q_\mathfrak{T} \in \mathcal{Q}(\mathfrak{p}')$. Then it holds that $10\mathfrak{p}_L \leq 10 \mathfrak{p}'$: This is obvious if $\mathfrak{p}_L = \mathfrak{p}'$. If $\mathfrak{p}_L \neq \mathfrak{p}'$ then $I_{\mathfrak{p}_L} \subsetneq I_{\mathfrak{p}'}$, which implies by \cref{ComparingPolBound} that for all $Q \in \mathcal{Q}(10 \mathfrak{p}')$
    \begin{align*}
        \|Q - Q_{\mathfrak{p}_L}\|_{I_{\mathfrak{p}_L}} 
        &\leq \|Q - Q_\mathfrak{T}\|_{I_{\mathfrak{p}_L}} + \|Q_\mathfrak{T} - Q_{\mathfrak{p}_L}\|_{I_{\mathfrak{p}_L}} \\
        &\leq 10^{-4}\|Q - Q_\mathfrak{T}\|_{I_{\mathfrak{p}'}} + \|Q_\mathfrak{T} - Q_{\mathfrak{p}_L}\|_{I_{\mathfrak{p}_L}}\\
        &\leq 10^{-4}(\|Q - Q_{\mathfrak{p}'}\|_{I_{\mathfrak{p}'}} + \|Q_{\mathfrak{p}'} - Q_\mathfrak{T}\|_{I_\mathfrak{p}'}) + \|Q_\mathfrak{T} - Q_{\mathfrak{p}_L}\|_{I_{\mathfrak{p}_L}}\\
        &\leq 10^{-4}(10 +1) + 1 \leq 10\,.
    \end{align*}
    Furthermore, every tile $\mathfrak{p} \in \mathfrak{T}$ with $L \cap I_\mathfrak{p} \neq \emptyset$ satisfies $10 \mathfrak{p}' \leq \mathfrak{p}$. Indeed, if $I_\mathfrak{p} \cap L \neq \emptyset$, then $I_{\mathfrak{p}'} = \hat L \subset I_\mathfrak{p}$ and for all $Q \in \mathcal{Q}(\mathfrak{p})$ it holds that
    \begin{align*}
        \|Q - Q_{\mathfrak{p}'}\|_{I_{\mathfrak{p}'}} &\leq \|Q - Q_\mathfrak{p}\|_{I_{\mathfrak{p}'}} + \|Q_\mathfrak{p} - Q_\mathfrak{T}\|_{I_{\mathfrak{p}'}} + \|Q_\mathfrak{T} - Q_{\mathfrak{p}'}\|_{I_{\mathfrak{p}'}} \\
        &\leq \|Q - Q_\mathfrak{p}\|_{I_{\mathfrak{p}}} + \|Q_\mathfrak{p} - Q_\mathfrak{T}\|_{I_{\mathfrak{p}}} + 4 \leq 1 + 4 + 4 \leq 10\,.
    \end{align*}
    Let $E(L) = \cup_{\mathfrak{p} \in \mathfrak{T}} E(\mathfrak{p})\cap L$. Since $10\mathfrak{p}' \leq \mathfrak{p}$ for all $\mathfrak{p} \in \mathfrak{T}$ with $E(\mathfrak{p}) \cap L \neq \emptyset$, it holds that $E(L) \subset \overline{E}(10 \mathfrak{p}')$. 
    Thus 
    \begin{align*}
        |E(L)| \leq |\overline{E}(10 \mathfrak{p}')| \leq 10^{\dim \mathcal{Q}} \dens(\mathfrak{p}_L) |I_{\mathfrak{p}'}| \lesssim \dens(\mathfrak{T}) |L|\,.
    \end{align*}
    We write $E = \cup_L E(L)$. Then \cref{ProjTreeBound} implies that 
    \begin{align*}
        \left| \int g T_\mathfrak{T} f \right| = \left| \int g\mathbf{1}_E T_\mathfrak{T} f \right|\lesssim A \|f\|_2 \|P_\mathcal{L}(\mathbf{1}_E g)\|_{2}\,.
    \end{align*}
    The estimates for the size of $E(L)$ allow us to obtain an improved bound for $P_\mathcal{L} \mathbf{1}_E$:
    \begin{align*}
        \|P_\mathcal{L}(\mathbf{1}_E g)\|_{2}^{2}
        &= \sum_L |L| \left(\frac{1}{|L|} \int_L \mathbf{1}_E |g| \right)^{2}\\
        &\leq \sum_L |L|  \left(\frac{1}{|L|}\int_L \mathbf{1}_E^2\right) \left(\frac{1}{|L|}\int_L |g|^{2} \right)\leq \dens(\mathfrak{T})\|g\|_2^2\,.
    \end{align*}
    This completes the proof.
\end{proof}

\subsection{Separated Trees}

\begin{definition}
    A tree $\mathfrak{T}$ is called \textit{normal} if for every $\mathfrak{p} \in \mathfrak{T}$ it holds that $I_\mathfrak{p}^* \subset I_\mathfrak{T}$.
\end{definition}
If $\mathfrak{T}$ is normal then $T_\mathfrak{T}^*f$ is supported in $I_\mathfrak{T}$ for all $f$, and if $\mathfrak{T}$ is any tree then $\mathfrak{T}' = \mathfrak{T} \setminus \bd(\mathfrak{T})$ is a normal tree.

We have the following improved estimate for $T_{\mathfrak{T}_1}T_{\mathfrak{T}_2}^*$ for separated normal trees $\mathfrak{T}_1$ and $\mathfrak{T}_2$:

\begin{lemma}[\cite{ZK2021}, Lem. 5.16]
\label{SepTreeBound}
 There exists $\varepsilon > 0$ such that for any two $\Delta$-separated normal trees $\mathfrak{T}_1$, $\mathfrak{T}_2$, and all $g_1, g_2 \in L^2(\mathbb{T}^\mathbf{d})$ it holds that 
 \begin{equation}
    \label{SepTreeBoundEqn}
    \left|\int_{\mathbb{T}^d} T_{\mathfrak{T}_1}^* g_1 \overline{ T_{\mathfrak{T}_2}^* g_2}\right| \lesssim  \Delta^{-\varepsilon} \prod_{j = 1,2} \||T^*_{\mathfrak{T}_j}g_j| + A Mg_j\|_{L^2(I_{\mathfrak{T}_1} \cap I_{\mathfrak{T}_2})}\,.
 \end{equation}
\end{lemma}

\begin{proof}
    The estimate holds without the factor $\Delta^{-\varepsilon}$, so we can assume that $\Delta \gg 1$. If $I_{\mathfrak{T}_1}$ and $I_{\mathfrak{T}_2}$ are disjoint the inequality is trivial, so we can further assume that $I_0 \coloneqq I_{\mathfrak{T}_1} \subset I_{\mathfrak{T}_2}$ and $\mathfrak{T}_1 \neq \emptyset$. We denote $Q = Q_{\mathfrak{T}_1} - Q_{\mathfrak{T}_2}$ and fix some $0 < \eta < 1$ which will be specified later. Let
    \[
        \mathfrak{S} = \{\mathfrak{p} \in \mathfrak{T}_1 \cup \mathfrak{T}_2 \, : \, \|Q\|_{I_\mathfrak{p}} \geq \Delta^{1- \eta}\}\,.
    \]
    The $\Delta$-separation of $\mathfrak{T}_1$ and $\mathfrak{T}_2$ implies that every $\mathfrak{p} \in \mathfrak{T}_1$ satisfies
    \[
        \|Q\|_{I_\mathfrak{p}} \geq \|Q_\mathfrak{p} - Q_{\mathfrak{T}_2}\|_{I_\mathfrak{p}} - \|Q_\mathfrak{p} - Q_{\mathfrak{T}_1}\|_{I_\mathfrak{p}} \geq \Delta - 1 - 4\,.
    \]
    The same holds if $\mathfrak{p} \in \mathfrak{T}_2$ with $I_\mathfrak{p} \subset I_0$, by the same argument. If $\mathfrak{p} \in \mathfrak{T}_2 \cup \mathfrak{T}_1$ with $I_0 \subset I_\mathfrak{p}$, then in particular $I_{\mathfrak{p}_1} \subset I_{\mathfrak{p}}$ for each $\mathfrak{p}_1 \in \mathfrak{T}_1$. Since $\mathfrak{T}_1$ is not empty, we can pick one such $\mathfrak{p}_1$, and since $\|Q\|_I$ is increasing in $I$, we conclude that $\|Q\|_{I_\mathfrak{p}} \geq \|Q\|_{I_{\mathfrak{p}_1}} \geq \Delta - 5$ also for such $\mathfrak{p}$. Hence if $\Delta$ is sufficiently large, then 
    \[
        I_\mathfrak{p} \cap I_0 = \emptyset \ \text{for all $\mathfrak{p} \in (\mathfrak{T}_1 \cup \mathfrak{T}_2) \setminus \mathfrak{S}$.}
    \]
    Define $\mathcal{J} = \{J \in \mathcal{J}(\mathfrak{S}) \, : \, J \subset I_0\}$. Since $\mathfrak{T}_1 \subset \mathfrak{S}$, there is no $J \in \mathcal{J}(\mathfrak{S})$ with $I_0 \subset J$. Hence $\mathcal{J}$ is a partition of $I_0$. We claim that there exists a partition of unity $\chi_J$, $J \in \mathcal{J}$ such that:
    \begin{enumerate}[label=(\arabic*)]
        \item \label{Partiton1} Each $\chi_J$ is smooth on $I_0$,
        \item \label{Partiton2} $\mathbf{1}_{I_0} = \sum_{J \in \mathcal{J}} \chi_J$,
        \item \label{Partiton3} $|\chi_J(x) - \chi_J(y)| \lesssim D^{-s(J)} \rho(x-y)$ for all $J \in \mathcal{J}$ and $x,y \in I_0$,
        \item \label{Partiton4} $\chi_J$ is supported in $N(J)$, where $N(J)$ is the union of all $D$-adic cubes $J'$ of scale $s(J) - 1$ whose closure intersects the closure of $J$.
    \end{enumerate}
    The functions $\chi_J$ will not be continuous outside of $I_0$ but this will not cause any problems. 
    
    To show the existence of the partition of unity, note that two cubes $J, J' \in \mathcal{J}$ with $\overline{J} \cap \overline{J}' \neq \emptyset$ have scales differing by at most one. Indeed, assume that $s(J') \geq s(J) + 2$. There is a cube $I_\mathfrak{p}$, $\mathfrak{p} \in \mathfrak{S}$ with $I_\mathfrak{p} \subset 100D\hat J$, by the definition of $\mathcal{J}(\mathfrak{S})$. Since $s(J') > s(\hat J)$ and $\overline{J} \cap \overline{J}' \neq \emptyset$, it follows that $I_\mathfrak{p} \subset 100 D J'$, a contradiction.
    Using this fact, the construction of such a partition is standard, see for example \cite{Grafakos2014}, Appendix J. 
    
    Set $\Delta_J = \|Q\|_J$ for $J \in \mathcal{J}$. For every $J \in\mathcal{J}$, there exists some $\mathfrak{p} \in \mathfrak{S}$ with $100D\hat J \supset I_\mathfrak{p}$. Thus 
    \begin{equation}
        \label{DeltaJEqn}
        \Delta_J = \|Q\|_J \gtrsim \|Q\|_{100 D \hat J} \gtrsim \|Q\|_{I_\mathfrak{p}} \geq \Delta^{1-\eta}
    \end{equation}
    for all $J \in \mathcal{J}$.
    
    We now consider the contribution of $\mathfrak{S}$. The strategy of the argument is to prove Hölder continuity of $e(-Q_\mathfrak{T}) T_\mathfrak{T}^* f$ for trees $\mathfrak{T}$ and then use \cref{VanDerCorputA1}.
    For every tile $\mathfrak{p} \in \mathfrak{T}$ and all $y, y' \in \mathbb{T}^\mathbf{d}$, it holds that
    \begin{align*}
        &|e(-Q_\mathfrak{T}(y)) T_\mathfrak{p}^* f(y) - e(-Q_\mathfrak{T}(y')) T_\mathfrak{p}^* f(y')|\\
        =&\bigg| \int e(-Q_x(x) + Q_x(y) - Q_\mathfrak{T}(y)) \overline{K_{s(\mathfrak{p})}(x-y)} (\mathbf{1}_{E(\mathfrak{p})} f)(x)\\ 
        &-  e(-Q_x(x) + Q_x(y') - Q_\mathfrak{T}(y')) \overline{K_{s(\mathfrak{p})}(x-y')} (\mathbf{1}_{E(\mathfrak{p})} f)(x) \, \mathrm{d}x\bigg|\\
        \leq&\int_{E(\mathfrak{p})} |f(x)| |e(Q_x(y) - Q_x(y') - Q_\mathfrak{T}(y) + Q_\mathfrak{T}(y'))\overline{K_{s(\mathfrak{p})}(x-y)} - \overline{K_{s(\mathfrak{p})}(x-y')}| \, \mathrm{d}x\\
        \leq&\int_{E(\mathfrak{p})} |f(x)| |e(Q_x(y) - Q_x(y') - Q_\mathfrak{T}(y) + Q_\mathfrak{T}(y')) - 1||\overline{K_{s(\mathfrak{p})}(x-y)}|\, \mathrm{d}x\\
        &+ \int_{E(\mathfrak{p})} |f(x)| |\overline{K_{s(\mathfrak{p})}(x-y)} - \overline{K_{s(\mathfrak{p})}(x-y')} |\, \mathrm{d}x
    \end{align*}
    Let $y, y' \in I_\mathfrak{p}^*$. Then we can estimate using \cref{PolynomialBound}
    \begin{align*}
        |Q_x(y) - Q_x(y') - Q_\mathfrak{T}(y) + Q_\mathfrak{T}(y')|
        &\leq \|Q_x - Q_\mathfrak{T}\|_{B_\rho(y, \rho(y-y'))} \\
        &\lesssim \left(\frac{\rho(y-y')}{D^{s(\mathfrak{p})}}\right)^{\alpha_1} \|Q_x - Q_\mathfrak{T}\|_{B_\rho(y, D^{s(\mathfrak{p})})}\\
        &\lesssim \left(\frac{\rho(y-y')}{D^{s(\mathfrak{p})}}\right)^{\alpha_1} \|Q_x - Q_\mathfrak{T}\|_{I_\mathfrak{p}} \lesssim \left(\frac{\rho(y-y')}{D^{s(\mathfrak{p})}}\right)^{\alpha_1}\,.
    \end{align*}
    Applying this and \eqref{KernelHoelderBound}, we obtain for $y,y' \in I_\mathfrak{p}^*$
    \begin{align*}
    |e(-Q_\mathfrak{T}(y)) T_\mathfrak{p}^* f(y) - e(-Q_\mathfrak{T}(y')) T_\mathfrak{p}^* f(y')|\lesssim A \frac{\rho(y-y')}{D^{s(\mathfrak{p})}}  D^{-s(\mathfrak{p})|\alpha|} \int_{E(\mathfrak{p})} |f(x)| \, \mathrm{d}x\,.
    \end{align*}
    Since $T_\mathfrak{p}^*f$ is supported in $I_\mathfrak{p}^*$ and vanishes on the boundary this estimate holds for all $y, y' \in \mathbb{T}^\mathbf{d}$.
    
    Next, suppose that $J \in \mathcal{D}$ has the property that
    \begin{equation}
        \label{SepTreesProp}
        \mathfrak{p} \in \mathfrak{T}, I_\mathfrak{p}^* \cap N(J) \neq \emptyset \implies s(\mathfrak{p}) \geq s(J)\,.
    \end{equation}
    Then, for $y, y' \in N(J)$, we have that:
    \begin{align}
        &\quad|e(-Q_\mathfrak{T}(y)) T_\mathfrak{T}^* f(y) - e(-Q_\mathfrak{T}(y')) T_\mathfrak{T}^* f(y')|\nonumber\\
        &\leq \sum_{\mathfrak{p} \in \mathfrak{T}, I_\mathfrak{p}^* \cap N(J) \neq \emptyset} |e(-Q_\mathfrak{T}(y)) T_\mathfrak{p}^* f(y) - e(-Q_\mathfrak{T}(y')) T_\mathfrak{p}^* f(y')|\nonumber\\
        &\lesssim A \rho(y-y') \sum_{s \geq s(J)} D^{-s(|\alpha| + 1)}\sum_{\mathfrak{p} \in \mathfrak{T}, s(\mathfrak{p}) = s, I_\mathfrak{p}^* \cap N(J) \neq \emptyset}  \int_{E(\mathfrak{p})} |f(x)|\nonumber\\
        &\lesssim A \rho(y-y') \sum_{s \geq s(J)} D^{-s}\inf_{J} Mf\nonumber\\
        &\lesssim A \frac{\rho(y - y')}{D^{s(J)}}\inf_{J} Mf\,.  \label{TreeHölder}
    \end{align}
    It follows that
    \begin{align}
        \label{TreeUB}
        \sup_{y \in N(J)} |T_\mathfrak{T}^*f(y)| \leq \inf_{y \in \frac{1}{2}J} |T^*_\mathfrak{T} f(y)| + CA \inf_{y\in J} Mf(y)\,.
    \end{align}
    We claim that there exists an absolute constant $s_0$ such that: 
    \begin{equation}
        \mathfrak{p} \in \mathfrak{T}_2 \setminus \mathfrak{S}, J \in \mathcal{J}, I_\mathfrak{p}^* \cap J \neq \emptyset \implies s(\mathfrak{p}) \leq s(J) + s_0\,.
    \end{equation}
    Suppose that $s(\mathfrak{p}) > s(J) + s_0$. By the definition of $\mathcal{J}$, there exists some $\mathfrak{p}' \in \mathfrak{S}$ with $I_{\mathfrak{p}'} \subset 100 D\hat J$. We further note that since $|D^{s_0-1}\hat J| < |I_\mathfrak{p}|$ and $I_\mathfrak{p}^* \cap J \neq \emptyset$, it follows that $D^{s_0 - 1} \hat{J} \subset 10I_{\mathfrak{p}}$.
    Hence
    \begin{align*}
        \Delta^{1- \eta} &> \|Q\|_{I_\mathfrak{p}} \gtrsim \|Q\|_{10 I_\mathfrak{p}} \geq \|Q\|_{D^{s_0 - 1}\hat J} \gtrsim D^{\alpha_1 s_0} \|Q\|_{100 D \hat J}\\
        &\geq D^{\alpha_1 s_0} \|Q\|_{I_{\mathfrak{p}'}} \geq D^{\alpha_1 s_0} \Delta^{1 - \eta}\,.
    \end{align*}
    Choosing $s_0$ large enough, we arrive at a contradiction and the claim follows.
    
    Recall that if $\mathfrak{p}\in \mathfrak{T}_2 \setminus \mathfrak{S}$, then $I_\mathfrak{p} \cap I_0 = \emptyset$. Thus if $I_\mathfrak{p}^* \cap \frac{1}{2}J \neq \emptyset$ for some $J \in \mathcal{J}$, then $s(\mathfrak{p}) \geq s(J)$. Using this and the claim, we obtain
    \begin{align*}
        \sup_{y \in \frac{1}{2}J} |T_{\mathfrak{T}_2 \setminus\mathfrak{S}}^* g(y)| &\leq \sup_{y \in \frac{1}{2}J} \sum_{\mathfrak{p} \in \mathfrak{T}_2\setminus \mathfrak{S}, I_\mathfrak{p}^* \cap J \neq \emptyset} |T_\mathfrak{p}^*g(y)|\\
        &\leq \sup_{y \in \frac{1}{2}J} \sum_{s = s(J)}^{s(J) + s_0} \sum_{\mathfrak{p} \in \mathfrak{P}, s(\mathfrak{p}) = s} |T_\mathfrak{p}^* g(y)|\\
        &\lesssim A \sum_{s = s(J)}^{s(J) + s_0} \sup_{y \in \frac{1}{2}J} D^{-s|\alpha|} \sum_{\mathfrak{p} \in \mathfrak{P}, s(\mathfrak{p}) = s} \int_{B_\rho(y, D^{s}/2)} |g\mathbf{1}_{E(\mathfrak{p})}| \\
        &\lesssim A (s_0 +1) \inf_{y \in J} Mg(y)\,.
    \end{align*}
    Note now that the set $\mathfrak{T}_2 \cap \mathfrak{S}$ is a tree with the property \eqref{SepTreesProp}: If $I_\mathfrak{p}^* \cap N(J) \neq \emptyset $ and $s(\mathfrak{p}) < s(J)$, then $I_\mathfrak{p} \subset 100 DJ$, contradicting the definition of $\mathcal{J}$. Thus we can apply \eqref{TreeUB} and obtain for all $J \in \mathcal{J}$:
    \begin{align}
        \sup_{y \in N(J)} |T^*_{\mathfrak{T}_2 \cap \mathfrak{S}} g(y)| &\leq \inf_{y \in \frac{1}{2}J} |T_{\mathfrak{T}_2 \cap \mathfrak{S}}^* g(y)| + CA\inf_{y \in J} Mg(y)\nonumber\\
        &\leq \inf_{y \in \frac{1}{2}J} |T^*_{\mathfrak{T}_2} g(y)| + \sup_{y \in \frac{1}{2}J}|T^*_{\mathfrak{T}_2\setminus \mathfrak{S}} g(y)| + CA \inf_{y \in J} Mg(y)\nonumber\\
        &\leq \inf_{y \in \frac{1}{2}J} |T^*_{\mathfrak{T}_2} g(y)| + CA \inf_{y \in J} Mg(y)\,.\label{TreeUB2}
    \end{align}
    Set 
    \[
        h_J(y) = \chi_J(y)(e(-Q_\mathfrak{T_1}(y)) T_{\mathfrak{T}_1}^* g_1(y)) \cdot \overline{(e(-Q_\mathfrak{T_2}(y)) T_{\mathfrak{T}_2 \cap \mathfrak{S}}^* g_2(y))}\,.
    \]
    Since $\mathfrak{T}_1$ is contained in $\mathfrak{S}$ it satisfies \eqref{SepTreesProp}, hence we can also apply \eqref{TreeHölder} and\eqref{TreeUB} to $\mathfrak{T} = \mathfrak{T}_1$. Combining this with \eqref{TreeHölder} for $\mathfrak{T} = \mathfrak{S} \cap \mathfrak{T}_2$ and \eqref{TreeUB2} as well as the Lipschitz estimate for $\chi_J$ yields for $y, y' \in N(J)$
    \begin{equation}
        \label{eq h Lip}
        |h_J(y)- h_J(y')| \lesssim  \frac{\rho(y - y')}{D^{s(J)}} \prod_{j = 1,2} (\inf_{\frac{1}{2}J} |T_{\mathfrak{T}_j}^* g_j| + A \inf_J Mg_j)(\mathbf{1}_{N(J)}(y) + \mathbf{1}_{N(J)}(y'))\,.
    \end{equation}
    Since $T_{\mathfrak{T}_i}^* g_i$ is continuous on $\mathbb{T}^\mathbf{d}$ for $i=1,2$, $T_{\mathfrak{T}_1}^* g_1$ vanishes outside of $I_0$ and $\chi_J$ is continuous (and bounded) on $I_0$ and vanishes outside of $N(J)$, the functions $h_J$ are continuous on $\mathbb{T}^\mathbf{d}$ and supported in $N(J)$. Thus the Lipschitz estimate \eqref{eq h Lip} holds in fact for all $y, y' \in \mathbb{T}^\mathbf{d}$. 
    
    We can finally apply \cref{VanDerCorputA1} (viewing the sets $N(J)$ as subsets of $\R^\mathbf{d}$):
    \begin{align*}
        \left| \int_{\mathbb{T}^\mathbf{d}} T_{\mathfrak{T}_1}^* g_1 \overline{T_{\mathfrak{T}_2 \cap \mathfrak{S}}^* g_2 }\right| 
        &\leq \sum_{J \in \mathcal{J}} \left|\int_{\R^\mathbf{d}} e(Q(y)) h_J(y) \, \mathrm{d}y \right|\\
        &\leq \sum_{J \in \mathcal{J}} \Delta_J^{-1/(d\alpha_\mathbf{d})} |J|\prod_{j = 1,2}(\inf_{\frac{1}{2}J} |T_{\mathfrak{T}_j}^* g_j| + A\inf_J Mg_j)\\
        &\leq  \Delta^{-(1 - \eta)/(d\alpha_\mathbf{d})} \prod_{j = 1,2}\| |T_{\mathfrak{T}_j}^* g_j| + A Mg_j\|_{L^2(I_0)}\,.
    \end{align*}
    The last step uses \eqref{DeltaJEqn}.
    
    It remains to estimate the contribution of $\mathfrak{T}_2 \setminus \mathfrak{S}$. Define $\mathcal{J}' = \{J \in \mathcal{J}(\mathfrak{T}_1) \, : \, J \subset I_0\}$. We claim that for some $s_\Delta$ with $D^{s_\Delta}\sim \Delta^{\eta/(d\alpha_\mathbf{d})}$ it holds that
    \begin{equation}
        \label{SepTreeClaim}
        \mathfrak{p} \in \mathfrak{T}_2 \setminus \mathfrak{S}, J \in \mathcal{J}', I_\mathfrak{p}^* \cap J \neq \emptyset \implies s(\mathfrak{p}) \leq s(J) - s_\Delta\,.
    \end{equation}
    Indeed, assume that $s(\mathfrak{p}) > s(J) - s_\Delta$. The cube $D^{s_\Delta}I_\mathfrak{p}$ is larger than $J$ and intersects it. Thus, if $C$ is choosen large enough, then $100D\hat J \subset CD^{s_\Delta}I_\mathfrak{p}$. On the other hand, there exists by the definition of $\mathcal{J}'$ some $\mathfrak{p}' \in \mathfrak{T}_1$ with $I_{\mathfrak{p}'} \subset 100 D \hat J$. This gives, by the definition of $\mathfrak{S}$ and \cref{PolynomialBound}:
    \begin{align*}
        \Delta^{1-\eta} > \|Q\|_{I_\mathfrak{p}} \gtrsim D^{-\alpha_\mathbf{d} ds_\Delta} \|Q\|_{CD^{s_\Delta}I_\mathfrak{p}} \gtrsim D^{-\alpha_\mathbf{d} ds_\Delta} \|Q\|_{I_{\mathfrak{p}'}} \geq D^{-\alpha_\mathbf{d} ds_\Delta}(\Delta - 2)\,.
    \end{align*}
    If the constant $c$ in $D^{s_\Delta} \leq c\Delta^{\eta/(d\alpha_\mathbf{d})}$ is chosen sufficiently small, this is a contradiction.
    
    Note that $\mathfrak{T}_2 \setminus \mathfrak{S}$ is still a tree, since $\mathfrak{S}$ is an up set.
    Thus, by \cref{ProjTreeBound} it holds that
    \begin{align*}
        \quad\left| \int T_{\mathfrak{T}_1}^* g_1 \overline{T_{\mathfrak{T}_2 \setminus \mathfrak{S}}^* g_2} \right|\lesssim A \|g_1\mathbf{1}_{I_0}\|_2 \|P_{\mathcal{J}'}|T_{\mathfrak{T}_2 \setminus \mathfrak{S}}^* g_2|\|_2\,.
    \end{align*}
    Using \eqref{SepTreeClaim} we have
    \begin{align*}
        \|P_{\mathcal{J}'}|T_{\mathfrak{T}_2 \setminus \mathfrak{S}}^* g_2|\|_2
        &\leq \sum_{s \geq s_\Delta} \left(\sum_{J \in \mathcal{J}'} |J|^{-1} \left|\int_J\sum_{\mathfrak{p} \in \mathfrak{T}_2 \setminus \mathfrak{S}: s(\mathfrak{p}) = s(J) - s, I_\mathfrak{p}^* \cap J \neq \emptyset} T_\mathfrak{p}^* g_2\right|^2\right)^{1/2}\\
        &\leq A \sum_{s \geq s_\Delta} \left(\sum_{J \in \mathcal{J}'} |J|^{-1} \left|\int_J Mg_2 \sum_{I \in \mathcal{D}_{s(J) - s}, I\cap I_0 = \emptyset, I^* \cap J \neq \emptyset}  \mathbf{1}_{I^*}\right|^2\right)^{1/2}\\
        &\leq A \sum_{s \geq s_\Delta} \left(\sum_{J \in \mathcal{J}'}  \int_J (Mg_2)^2  |J|^{-1}\int_J\left(\sum_{I \in \mathcal{D}_{s(J) - s}, I\cap I_0 = \emptyset, I^* \cap J \neq \emptyset}  \mathbf{1}_{I^*}\right)^2 \right)^{1/2}\,.
    \end{align*}
    The cubes 
    \[
        \{I^* \, : \, I \in \mathcal{D}_{s(J) - s}, I\cap I_0 = \emptyset, I^* \cap J \neq \emptyset\}
    \]
    have bounded overlap and cover a set of size $\lesssim |J| D^{-\alpha_1s}$. Thus we can further estimate the $s$-sum by
    \begin{align*}
         \sum_{s \geq s_\Delta} D^{-\alpha_1 s/2} \|\mathbf{1}_{I_0}Mg_2\|_2 \lesssim D^{-\alpha_1 s_\Delta/2} \|\mathbf{1}_{I_0}Mg_2\|_2 \lesssim \Delta^{-\alpha_1\eta/(2d\alpha_\mathbf{d})} \|\mathbf{1}_{I_0}Mg_2\|_2\,.
    \end{align*}
    This completes the estimate of the contribution of $\mathfrak{T}_2 \setminus \mathfrak{S}$. Finally, set $\eta = 2/(2 + \alpha_1)$ to obtain \eqref{SepTreeBoundEqn} with $\varepsilon = \alpha_1/((2 + \alpha_1)\alpha_\mathbf{d}d)$.
\end{proof}

\subsection{Rows of Trees}

\begin{definition}
    A \textit{row} is a union of normal trees with pairwise disjoint spatial cubes.
\end{definition}

\Cref{SepTreeBound} implies an estimate for $T_{\mathfrak{R}_1}T_{\mathfrak{R}_2}^*$ for separated rows $\mathfrak{R}_1$, $\mathfrak{R}_2$ with decay in the separation:

\begin{lemma}[\cite{ZK2021}, Lem. 5.26]
    \label{SepRowBound}
    Let $\mathfrak{R}_1$, $\mathfrak{R}_2$ be rows such that the trees in $\mathfrak{R}_1$ are $\Delta$-separated from the trees in $\mathfrak{R}_2$. Then for any $g_1$, $g_2 \in L^2(\mathbb{T}^\mathbf{d})$, it holds that
    \[
        \left| \int_{\mathbb{T}^\mathbf{d}} T_{\mathfrak{R}_1}^* g_1 \overline{T_{\mathfrak{R}_2}^* g_2} \right| \lesssim A^2 \Delta^{-\varepsilon} \|g_1\|_2 \|g_2\|_2\,.
    \]
\end{lemma}

\begin{proof}
    Denote by $S_\mathfrak{T} $ the operator $|T_\mathfrak{T}^*| + AM$. These operators are bounded on $L^2$ with norm $\lesssim A$ for all trees $\mathfrak{T}$, by \cref{ProjTreeBound}. Using this and \cref{SepTreeBound}, we obtain:
    \begin{align*}
        \left| \int T^*_{\mathfrak{R}_1}g_1 \overline{T_{\mathfrak{R}_2}^* g_2}\right|
        &\leq \sum_{\mathfrak{T}_1 \in \mathfrak{R}_1, \, \mathfrak{T}_2 \in \mathfrak{R}_2} \left| \int T^*_{\mathfrak{T}_1}g_1 \overline{T_{\mathfrak{T}_2}^* g_2}\right|\\
        &\leq \Delta^{-\varepsilon} \sum_{\mathfrak{T}_1 \in \mathfrak{R}_1, \, \mathfrak{T}_2 \in \mathfrak{R}_2} \|S_{\mathfrak{T}_1} \mathbf{1}_{I_{\mathfrak{T}_1}}g_1\|_{L^2(I_{\mathfrak{T}_1} \cap I_{\mathfrak{T}_2})}\|S_{\mathfrak{T}_2} \mathbf{1}_{I_{\mathfrak{T}_2}}g_2\|_{L^2(I_{\mathfrak{T}_1} \cap I_{\mathfrak{T}_2})}\\
        &\leq \Delta^{-\varepsilon} \prod_{j=1,2} \left(\sum_{\mathfrak{T}_1 \in \mathfrak{R}_1, \, \mathfrak{T}_2 \in \mathfrak{R}_2} \|S_{\mathfrak{T}_j} \mathbf{1}_{I_{\mathfrak{T}_j}}g_j\|_{L^2(I_{\mathfrak{T}_1} \cap I_{\mathfrak{T}_2})}^2\right)^{1/2}\\
        &\leq \Delta^{-\varepsilon} \prod_{j=1,2} \left(\sum_{\mathfrak{T}_j \in \mathfrak{R}_j} \|S_{\mathfrak{T}_j} \mathbf{1}_{I_{\mathfrak{T}_j}}g_j\|_{L^2(I_{\mathfrak{T}_j})}^2\right)^{1/2}\\
        &\lesssim A^2 \Delta^{-\varepsilon} \left(\sum_{\mathfrak{T}_1 \in \mathfrak{R}_1} \|\mathbf{1}_{I_{\mathfrak{T}_1}}g_1\|_2^2\right)^{1/2}\left(\sum_{\mathfrak{T}_2 \in \mathfrak{R}_2} \|\mathbf{1}_{I_{\mathfrak{T}_2}}g_1\|_2^2\right)^{1/2}\\
        &\leq A^2 \Delta^{-\varepsilon} \|g_1\|_2 \|g_2\|_2\,. 
    \end{align*}
    Here the third step follows from the Cauchy-Schwarz inequality and the fourth and last step follow from the disjointness of the spatial cubes $I_\mathfrak{T}$ for $\mathfrak{T} \in \mathfrak{R}_j$.
\end{proof}

\subsection{Forests}
Finally, we estimate the contribution of forests of normal trees using an orthogonality argument. All trees can be assumed to be normal, since we already estimated the contribution of boundary parts of trees in \cref{BoundaryTreesSection}. We fix $n$ and drop it from the notation.
\begin{lemma}[\cite{ZK2021}, Prop. 5.27]
    \label{ForestBound}
    Let $\mathfrak{N}_{j,l} = \mathfrak{T}_{j,l} \setminus \bd(\mathfrak{T}_{j,l})$ and set $\mathfrak{F}'_{j} = \cup_l \mathfrak{N}_{j,l}$. If $n > n_0$, then it holds that
    \[
        \|T_{\mathfrak{F}'_{j}}\|_{2\to 2} \lesssim A 2^{-n/2}
    \]
    and if $n = n_0 \geq \frac{2}{\log(2)}\log(A/(\|R^K\|_{2 \to 2} + \|M^K\|_{2 \to 2}))$, then it holds that
    \[
        \|T_{\mathfrak{F}'_{j}}\|_{2\to 2} \lesssim \|R^K\|_{2 \to 2} + \|M^K\|_{2 \to 2}\,.
    \]
\end{lemma}

\begin{proof}
    We subdivide the forest $\mathfrak{F}'_{j}$ into rows $\mathfrak{R}_m$ using the following procedure: If $\mathfrak{R}_1, \dots, \mathfrak{R}_i$ have been selected, $\mathfrak{R}_{i+1}$ is a maximal union of trees $\mathfrak{N}_{j,l}$ such that the cubes $I_{\mathfrak{N}_{j,l}}$ are maximal among the spatial cubes of not yet selected trees. Since the cubes $I_{\mathfrak{N}_{j,l}}$ have overlap bounded by $2^n \log (n+1) \log(\lambda)$ we obtain $\lesssim 2^{2n} \log(\lambda)$ rows. By \cref{TrivialTreeBound} and \cref{SumTreeBound}, it holds that
    \[
        \|T_{\mathfrak{N}_{j,l}} \|_{2 \to 2} \lesssim  
        \begin{cases}
        A 2^{-n/2} &\text{if $n > n_0$}\\
        \|R^K\|_{2 \to 2} + \|M^K\|_{2 \to 2} &\text{if $n = n_0$}
        \end{cases} 
        \eqqcolon C(n)\,.
    \]
    We have that $T_{\mathfrak{N}_{j,l}}^* f = T_{\mathfrak{N}_{j,l}}^* (\mathbf{1}_{I_{\mathfrak{N}_{j,l}}}f)$. Hence, by the disjointness of the spatial cubes in a row:
    \begin{align*}
        \|T_{\mathfrak{R}_m}^* f\|^2_2 = \sum_l \|T_{\mathfrak{N}_{j,l}}^* (\mathbf{1}_{I_{\mathfrak{N}_{j,l}}}f)\|_2^2\lesssim C(n)^2 \sum_l \|\mathbf{1}_{I_{\mathfrak{N}_{j,l}}}f\|_2^2\leq C(n)^2 \|f\|_2^2\,.
    \end{align*}
    It follows that for all $m$ we have $\|T_{\mathfrak{R}_m}\|_{2\to 2} \lesssim C(n)$.
    Note that the sets $E(\mathfrak{p})$ for $\mathfrak{p}$ in different rows are disjoint, since the rows are separated.
    This implies that $T^*_{\mathfrak{R}_m} T_{\mathfrak{R}_{m'}} = 0$ for $m \neq m'$. Thus the subspaces $\overline{\im T_{\mathfrak{R}_m}}$ for distinct $m$ are orthogonal. Denote by $\pi_m$ the orthogonal projection onto $\overline{\im T_{\mathfrak{R}_m}}$, so that $T^*_{\mathfrak{R}_m} = T^*_{\mathfrak{R}_m} \circ \pi_m$ for each $m$. Then
    \begin{align*}
        \|\sum_m T^*_{\mathfrak{R}_m} f\|_2^2  &=  \|\sum_m T^*_{\mathfrak{R}_m} \pi_m f\|_2^2 \\
        &=  \sum_m \|T^*_{\mathfrak{R}_m} \pi_m f\|_2^2 + \sum_{m \neq m'} \int \pi_m (f) T_{\mathfrak{R}_m} T^*_{\mathfrak{R}_{m'}} \pi_{m'}(f) \, \mathrm{d}x\\
        &\lesssim  C(n)^2 \sum_m \|\pi_m f\|_2^2 + A^2 2^{-\gamma n \varepsilon}\sum_{m \neq m'} \|\pi_m f\|_2 \|\pi_{m'} f\|_2 \\
        &\lesssim C(n)^2 \|f\|_2^2 + A^2 2^{2n}\log(\lambda) 2^{-\gamma n \varepsilon} \|f\|_2^2\,.
    \end{align*}
    Here we used the $2^{\gamma n}$-separation of the rows $\mathfrak{R}_m$ and \cref{SepRowBound} to estimate the $m \neq m'$ terms. The last inequality holds since there are $\lesssim 2^{2n} \log(\lambda)$ rows. Choosing $\gamma = \log \log(\lambda)/(\varepsilon \log 2) + 3/\varepsilon$, this is smaller than
    \[
        (C(n)^2 + A^22^{-n}) \|f\|_{2}^2\,.
    \]
    In both cases $n = n_0$ and $n > n_0$, this implies the claimed estimates.
\end{proof}

\section{Proof of the Main Theorem}
\label{PMain}
We now prove \cref{MainThmWeakBound1} using the results from Sections 2 to 4.

\begin{proof}[Proof of \cref{MainThmWeakBound1}]
    Fix $f \in L^2(\mathbb{T}^\mathbf{d})$ with $\|f\|_2 \leq 1$ and $\lambda > 10e$.
    By the reductions in Section 2, it suffices to estimate $T_\mathfrak{P}f$. We do this by showing bounds for the distribution function $|\{|T_\mathfrak{P}f(x)|>\lambda\}|$. 
    
    By \cref{FirstExceptionalSet}, there exists an exceptional set $E_1$ with $|E_1| \lesssim \lambda^{-2}$ and
    \begin{equation}
    \label{deceqn}
       \mathbf{1}_{\mathbb{T}^\mathbf{d} \setminus E_1} T_\mathfrak{P} = \mathbf{1}_{\mathbb{T}^\mathbf{d} \setminus E_1} \sum_{n \geq n_0} \left( \sum_{j = 1}^{C(\gamma n^2 + \gamma n \log\log \lambda)} T_{\mathfrak{A}_{n,j}} + \sum_{j = 1}^{C(n + \log\log\lambda)} T_{\mathfrak{F}_{n,j}} \right)\,.
    \end{equation}
    \Cref{TrivialAntichainBound} yields that $\|T_{\mathfrak{A}_{n,j}}\|_{2 \to 2} \lesssim \|M^K\|$ for $n = n_0$ and \cref{AntichainBound} implies the estimate $\|T_{\mathfrak{A}_{n,j}}\|_{2 \to 2} \lesssim  A 2^{-n\varepsilon}$ for $n > n_0$. We choose $n_0$ as the closest integer to 
    \[
        100\varepsilon^{-1} \log(e + \frac{A}{\|M^K\| + \|R^K\|})\,.
    \]
    It holds that $\|R^K\| + \|M^K\| \lesssim A$, by \cref{nontangential} and since $M^K \leq A M$.
    Using this and $\gamma \sim \log \log \lambda$, we find
    \[
        \left\|\sum_{n \geq n_0} \sum_{j = 1}^{C(\gamma n^2 + \gamma n \log\log\lambda)} T_{\mathfrak{A}_{n,j}}\right\|_{2\to2} \lesssim (\|M^K\|+ \|R^K\|) \log^2(e + \frac{A}{\|M^K\| + \|R^K\|}) (\log \log \lambda)^2\,.
    \]
    Next, we split the second summand in \eqref{deceqn} into forests of normal trees and boundary parts of trees:
    \[
        \sum_{j = 1}^{C(n+\log\log \lambda)} T_{\mathfrak{F}_{n,j}} = \sum_{j = 1}^{C(n+\log\log \lambda)} T_{\mathfrak{F}_{n,j}'}+ \sum_{j=1}^{C(n+\log\log \lambda)} T_{\cup_l\bd\mathfrak{T}_{n,j,l}}\,.
    \]
    For the boundary parts of trees we apply \cref{BoundaryTrees}. We obtain exceptional sets $E_2(n,j)$ such that $|E_2(n,j)| \lesssim 2^{-n} \lambda^{-2}$ and 
    \[
        \|\mathbf{1}_{\mathbb{T}^\mathbf{d} \setminus E_2(n,j)}T_{\cup_l\bd\mathfrak{T}_{n,j,l}}\|_{2 \to 2} \lesssim 
        \begin{cases}
            n \log(\lambda) A  2^{-n\varepsilon} & \text{if $n > n_0$}\\
            n \log(\lambda) \|M^K\| & \text{if $n = n_0$}
        \end{cases}\,.
    \]
    Then the set $E_2 \coloneqq \cup_{n,j} E_2(n,j)$ satisfies $|E_2| \lesssim \lambda^{-2} \log\log\lambda$ and
    \[
        \left\| \mathbf{1}_{\mathbb{T}^\mathbf{d} \setminus E_2}\sum_{n \geq n_0} \sum_{j = 1}^{C(n + \log\log\lambda)} T_{\cup_l\bd\mathfrak{T}_{n,j,l}}\right\|_{2\to2} \lesssim (\|M^K\| + \|R^K\|) \log^2(e + \frac{A}{\|M^K\| + \|R^K\|}) \log^2 \lambda \,.
    \]
    By \cref{ForestBound} it holds that
    \[
        \|T_{\mathfrak{F}_{n,j}'}\|_{2 \to 2} \lesssim 
        \begin{cases}
            A 2^{-n/2} &\text{if $n > n_0$}\\
            \|R^K\| + \|M^K\| &\text{if $n = n_0$}
        \end{cases}\,.
    \]
    Summing these estimates yields
    \begin{align*}
        \left\|\sum_{n \geq n_0} \sum_{j = 1}^{C(n+\log\log\lambda)} T_{\mathfrak{F}'_{n,j}}\right\|_{2 \to 2}&\lesssim (n_0 + \log\log\lambda)(\|R^K\| + \|M^K\|+ A 2^{-n_0/2})\\
        &\lesssim (\log \log \lambda  + \log(e + \frac{A}{\|M^K\| + \|R^K\|})) (\|M^K\| + \|R^K\|)\,.
    \end{align*}
    Putting everything together, we obtain with Tschebyscheff's inequality
    \begin{align*}
        |\{x \, : \, |T_\mathfrak{P}f(x)| > \lambda\}|\lesssim \frac{\log^4 \lambda}{\lambda^2}(1 + (\|M^K\| + \|R^K\|)^2 \log^4(e + \frac{A}{\|M^K\| + \|R^K\|}))\,.
    \end{align*}
    By multiplying the kernel $K$ with a constant, we can assume without loss of generality that $(\|R^K\|_{2 \to2} + \|M^K\|_{2 \to 2}) \log^2(e  +\frac{A}{\|M^K\| + \|R^K\|}) = 1$. Integrating the above estimate then yields $\|T^\mathfrak{P}f\|_p \lesssim_p 1$ for $p < 2$, as required.
\end{proof}

\appendix

\section{Appendix: Applications of \texorpdfstring{\cref{MainThmWeakBound1}}{Theorem 1.1}}

\subsection{Weak type \texorpdfstring{$L^2$}{L2} estimates}
We define for $f \in L^2(\R^\mathbf{d})$ the maximal modulation operator on $\R^\mathbf{d}$
\begin{align}
    \label{TRDef}
    T_\mathbb{R}^\mathcal{Q}f(x) = \sup_{Q \in \mathcal{Q}} \sup_{0 < \underline{R} < \overline{R}} \left| \int_{\underline{R} < \rho(x-y) < \overline{R}} K(x-y) e^{i Q(y)} f(y) \, \mathrm{d}y \right|\,.
\end{align}

\begin{corollary}
    \label{MainThmAnisotopic}
    Let $K$ be a Calderón-Zygmund with anisotropic scaling, with constant $A$.
    Then the operator $T^\mathcal{Q}$ is bounded from $L^2(\mathbb{T}^\mathbf{d})$ into $L^p(\mathbb{T}^\mathbf{d})$ for $1 \leq p < 2$ with norm
    \[
        \|T^\mathcal{Q}\|_{2 \to p} \lesssim_{\alpha, d,p} A \,.
    \]
    The operator $T^\mathcal{Q}$ is bounded from $L^2(\mathbb{T}^\mathbf{d})$ into $L^{2,\infty}(\mathbb{T}^\mathbf{d})$ and the operator $T_\R^\mathcal{Q}$ is bounded from $L^2(\R^\mathbf{d})$ into $L^{2, \infty}(\R^\mathbf{d})$.
\end{corollary}

\begin{remark}
    It is also possible to prove strong estimates on $L^p$, for $1 < p < \infty$, by adapting the arguments in \cite{ZK2021} to the anisotropic setting. 
    With this approach one obtains $\|T^\mathcal{Q}\|_{p \to p} \lesssim A$ and $\|T_\R^\mathcal{Q}\|_{p \to p} \lesssim A$. 
\end{remark}

\begin{proof}[Proof of \cref{MainThmAnisotopic}]
    Let $\psi \in C_c^\infty((1/8, 1/2))$ with $\sum_{k \in \mathbb{Z}} \psi(2^k t) = 1$ for all $t > 0$. For every Calderón-Zygmund kernel $K$, the kernels $K_s(x) = \psi(2^{-s} \rho(x)) K(x)$ form an admissible decomposition of $K$. It holds that $M^K \lesssim A M$, where $M$ is the Hardy-Littlewood maximal functions, and $\|R^K\|_{2 \to 2} \lesssim A$, see \cref{nontangential}. Thus \cref{MainThmWeakBound1} implies that for $1 \leq p < 2$
    \[
        \|T^\mathcal{Q}\|_{2 \to p} \lesssim (\|M^K\|_{2 \to 2} + \|R^K\|_{2 \to 2})\log^2(e + \frac{A}{\|M^K\|_{2 \to 2} + \|R^K\|_{2 \to 2}}) \lesssim A\,,
    \]
    where the last inequality holds because $x\log^2(e + A/x)$ is increasing in $x > 0$ for all $A > 0$.
    
    The integral in the definition of $T^\mathcal{Q}$ is continuous in the polynomial $Q$ and in the truncation parameters $\underline{R}$ and $\overline{R}$, thus the operator $T^\mathcal{Q}$ does not change when the supremum is restricted to a countable dense set of polynomials and truncations. For fixed $Q$, $\underline{R}$ and $\overline{R}$, the operator without suprema is a convolution operator with bounded kernel, thus it is translation invariant, linear, and bounded on $L^2(\mathbb{T}^\mathbf{d})$. Hence the Stein maximum principle (see \cite{Stein1961}, Cor. 1) implies that 
    there exists a constant $C$ such that
    \[
        \|T^\mathcal{Q} f\|_{L^{2,\infty}(\mathbb{T}^\mathbf{d})} \leq C \|f\|_{L^2(\mathbb{T}^\mathbf{d})}
    \]
    for every $f \in L^2(\mathbb{T}^\mathbf{d})$.
    
    It remains to show that $\|T_\R^\mathcal{Q} f\|_{L^{2,\infty}(\mathbb{R}^\mathbf{d})} \leq C' \|f\|_{L^2(\mathbb{R}^\mathbf{d})}$ for all $f \in L^2(\R^\mathbf{d})$. By density, we may assume that $f$ is compactly supported.
    Since the operator $T_\R^\mathcal{Q}$ is translation invariant and the conditions of \cref{MainThmAnisotopic} are invariant under conjugation by anisotropic dilations, we may further assume that $f$ is supported in $[-1/4, 1/4]^\mathbf{d}$. 
    Write
    \begin{align*}
        T_\R^\mathcal{Q}f(x) 
        &\leq \sup_{Q \in \mathcal{Q}} \sup_{0 < \underline{R} < \overline{R} < 1/8} \left| \int_{\underline{R} < \rho(x-y) < \overline{R}} K(x-y) e^{i Q(y)} f(y) \, \mathrm{d}y \right| \\
        &+  \int_{1/8 < \rho(x-y)} |K(x-y)| |f(y)| \, \mathrm{d}y\,.
    \end{align*}
    The first term is only nonzero if $\rho(x-y) \leq 1/8$ for some $y \in \spt f \subset [-1/4, 1/4]^\mathbf{d}$. If $\rho(x-y) \leq 1/8$ then $|x-y| \leq 1/8$. Hence the first term is supported in $[-1/2, 1/2)^\mathbf{d}$ and there it is dominated by $T^\mathcal{Q} f(x)$. By the support assumption on $f$, there exists a constant $C_\mathbf{d}$ and some $r_0 > 0$ such that the integrand in the second term is supported in $\{y \, : \, r_0 \leq \rho(x-y) \leq C_\mathbf{d} r_0\}$.
    Hence the second term is bounded by $A\,Mf$, where $M$ is the Hardy-Littlewood maximal function. We conclude that $T^\mathcal{Q}_\mathbb{R}$ is bounded from $L^2(\R^\mathbf{d})$ into $L^{2,\infty}(\R^\mathbf{d})$.
\end{proof}

\subsection{Singular Integrals Concentrated Near a Parabola}
\cref{MainThmWeakBound1} implies boundedness of maximal modulations of certain singular integral operators with very rough kernels. Let $\phi$ be a continuously differentiable function on $[-1,1]\setminus \{0\}$ which vanishes outside $[-1/4,1/4]$ and such that for some $\varepsilon > 0$
\begin{align}
    \label{MTWB2Condititon}
    |\phi(s)| &\lesssim \frac{1}{|s| \lvert\log \lvert s\rvert \rvert^{3 + \varepsilon}} & &\text{and} & |\phi'(s)| &\lesssim \frac{1}{|s|^2 \lvert\log \lvert s\rvert \rvert^{3 + \varepsilon}}\,.
\end{align}
Define a tempered distribution $K$ on $\R^2$ by
\begin{equation}
    \label{RoughKernel}
    K f = \lim_{\varepsilon \to 0} \int_{|x| > \varepsilon} \int f(x,y) \phi(\frac{y}{x^2} - 1) \, \mathrm{d}y \, \frac{\mathrm{d}x}{x^3} \,,
\end{equation}
and let $Tf = K* f$ for Schwartz functions $f$. By the boundedness of $H_P$ on $L^2(\R^2)$, the operator $T$ is bounded on $L^2(\R^2)$. Since the distribution $K$ agrees with a locally integrable function on $\R^2 \setminus \{0\}$, we can still define the maximally polynomially modulated, maximally truncated operators $T^\mathcal{Q}$ and $T_\mathbb{R}^\mathcal{Q}$ using \eqref{TTDef} and \eqref{TRDef}. As $K$ is not locally in $L^2$, we define them a priori only on smooth functions on the torus and Schwartz functions on $\R^2$. 

Note that one obtains the kernel of the Hilbert transform along the parabola from \eqref{RoughKernel} by setting (formally) $\phi = \delta_0$. Thus $T$ is a slightly less singular version of the Hilbert transform along the parabola. We have the following new result about maximal polynomial modulations of $T$:
\begin{corollary}
    \label{MainThmWeakBound2}
    Suppose that $\phi$ and $K$ are as above.
    Then the operator $T^\mathcal{Q}$ defined on smooth functions by \eqref{TTDef} extends to a bounded operator from $L^2(\mathbb{T}^2)$ into $L^{2,\infty}(\mathbb{T}^2)$. The operator $T_\R^\mathcal{Q}$ defined on Schwartz functions by \eqref{TRDef} extends to a bounded operator from $L^2(\R^2)$ into $L^{2,\infty}(\R^2)$.
\end{corollary}

\begin{remark}
    In \cref{MainThmWeakBound2} the parabola can be replaced by any homogeneous curve. In general, it holds that if $R^K$ and $M^K$ are bounded on $L^2(\R^\mathbf{d})$ for a singular integral operator on a submanifold, then one can deduce an analogue of \cref{MainThmWeakBound2}.
\end{remark}

\begin{proof}[Proof of \cref{MainThmWeakBound2}]
    We decompose the kernel $K$. Let $\psi \in C_c^\infty((1/8, 1/4 + \delta))$ with $\psi\geq 0$ and $\sum_{s\in\mathbb{Z}} \psi(2^st) = 1$ for all $t > 0$, where $\delta$ will be fixed later. Define $\phi_j(x) = \psi(2^{-j}|x|)\phi(x)$, this is not zero only when $j \leq 0$, and set $K^j(x,y) = \phi_j(\frac{y}{x^2} - 1) \frac{1}{x^3}$, so that $K = \sum_{j \leq 0} K^j$. 
    Our assumption \eqref{MTWB2Condititon} implies that for all $x$ and $j$
    \begin{align*}
        |\phi_j(x)| \lesssim 2^{-j} |j-1|^{-3 - \varepsilon}\,,\quad |\phi_j'(x)| \lesssim 2^{-2j} |j-1|^{-3 - \varepsilon}\,.
    \end{align*}
    Thus $K^j$ is a Calderón-Zygmund kernel with anisotropic scaling with constant $A(j) \lesssim 2^{-2j}|j-1|^{-3-\varepsilon}$ for all $j$. We further set
    \[
        K^j_s(x,y) = \phi_j(\frac{y}{x^2} - 1) \psi(2^{-s}x) \frac{1}{x^3}\,.
    \]
    For $\delta > 0$ sufficiently small, this defines an admissible decomposition of $K^j$.
    Next, we deduce from \eqref{MTWB2Condititon} that $\|\phi_j(x)\|_1 \lesssim |j-1|^{-3-\varepsilon}$. From the boundedness of the maximal functions associated to the Hilbert transform along the parabola it follows that
    \begin{align*}
        \|M^{K^j}\|_{2 \to 2} \sim \|\phi_j\|_1 \lesssim |j-1|^{-3-\varepsilon}\,,\quad \|R^{K^j}\|_{2 \to 2} \lesssim \|\phi_j\|_1 \lesssim |j-1|^{-3-\varepsilon}\,.
    \end{align*}
    Now we apply \cref{MainThmWeakBound1}. 
    Denoting by $(T^j)^\mathcal{Q}$ the maximal modulation operator on the torus associated to $K^j$, we obtain
    \begin{align*}
        \|(T^j)^\mathcal{Q}\|_{2 \to p} \lesssim  |j-1|^{-3-\varepsilon} \log^2(e + 2^{-2j})\lesssim |j - 1|^{-1-\varepsilon}\,.
    \end{align*}
    This is summable in $j$, hence $\|T^\mathcal{Q}\|_{2 \to p} < \infty$.
    The weak type $L^2$ estimates on the torus and on $\R^2$ follows from this estimate, using similar arguments as in the proof of \cref{MainThmAnisotopic}.
\end{proof}

\subsection{Singular Integrals with Weak Continuity Assumptions}
In a different direction, \cref{MainThmWeakBound1} can be used to weaken the continuity requirements of the kernel $K$ in \cref{MainThmAnisotopic}, as long as it is homogeneous and odd. Define a modulus of continuity to be an increasing function $\omega: [0,\infty) \to [0,\infty)$ such that $\lim_{x \to 0} \omega(x) = \omega(0) = 0$ and $\omega(x + y) \leq \omega(x) + \omega(y)$ for all $x,y$. In addition, we will assume that it satisfies 
\begin{equation}
    \label{StrongDini}
    \|\omega\| = \int_0^1 \omega(t) \log^2(\frac{1}{t}) \, \frac{\mathrm{d}t}{t} < \infty\,.
\end{equation}
Note that this condition is slightly stronger than the Dini condition $\int_0^1 \omega(t) \, \mathrm{d}t/t < \infty$. The Dini condition occurs in the theory of singular integrals as a sufficient condition on the modulus of continuity of an $L^{p_0}$ bounded Calderón-Zygmund kernel to ensure $L^p$ boundedness for $1 < p \leq p_0$, see \cite{Stein1993} I.6.3..

Let $\phi: S^{1} \to \mathbb{C}$ be an odd function such that for all $x, x' \in S^{1}$
\[ 
    |\phi(x) - \phi(x')| \leq \omega(|x-x'|)\,.
\]
Let $\alpha = (1,2)$. Define a kernel $K$ on $\R^2$ by
\[
    K(x) = \frac{1}{\rho(x)^{|\alpha|}} \phi(\delta_{\rho(x)^{-1}}(x))\,.
\]
Since $\phi$ is odd, the kernel $K$ has integral $0$ over all sets $\{\underline{R} \leq \rho(x) \leq \overline{R}\}$. This implies that for all Schwartz functions $f$ the limit
\[
    Kf = \lim_{\varepsilon \to 0} \int_{\rho(x) > \varepsilon} f(x) K(x) \, \mathrm{d}x 
\]
exists and defines a tempered distribution.
Let $T$ be the operator defined on Schwartz functions by $Tf = K * f$. Then $T$ extends to a bounded operator on $L^2(\R^2)$ since it is a superposition of Hilbert transforms along the curves $(t, c\sgn(t)t^2)$. Since $K$ is continuous on $\R^2\setminus\{0\}$, we can define the operators $T^\mathcal{Q}$ and $T_\mathbb{R}^\mathcal{Q}$ for $L^2$ functions using \eqref{TTDef} and \eqref{TRDef}. \Cref{MainThmWeakBound1} implies the following new result:

\begin{corollary}
\label{MainThmDini}
    Suppose that $\phi$ and $K$ are as above.
    Then the operator $T^\mathcal{Q}$ defined by \eqref{TTDef} is bounded from $L^2(\mathbb{T}^2)$ into $L^p(\mathbb{T}^2)$ for $1 \leq p < 2$ with norm
    \[
        \|T^\mathcal{Q}\|_{2 \to p} \lesssim_{d,p} \|\omega\|\,.
    \]
    The operator $T^\mathcal{Q}$ is bounded from  $L^2(\mathbb{T}^2)$ into $L^{2, \infty}(\mathbb{T}^2)$ and the operator $T_\R^\mathcal{Q}$ defined by \eqref{TRDef} is bounded from $L^2(\R^2)$ into $L^{2,\infty}(\R^2)$.
\end{corollary}

\begin{remark}
    \Cref{MainThmDini} should be compared to \cref{MainThmAnisotopic}, which has stronger continuity requirements, but does not require the kernel to be homogeneous and odd.
    The corollary is formulated on $\R^2$ with $\alpha = (1,2)$ only to simplify the notation, it is still true in higher dimension and for all $\alpha$, with a similar proof. 
    Note also that in the isotropic case this theorem holds for all integrable $\phi: \mathbb{S}^{\mathbf{d}-1} \to \mathbb{C}$, without any continuity assumptions. This follows from the polynomial Carleson theorem, since $K$ can then be written as a superposition of Hilbert transforms along lines.
\end{remark}

\begin{remark}
    We briefly explain where the oddness and homogeneity assumptions are used. In the proof of \cref{MainThmDini}, we will decompose the kernel $K$ into smoother pieces $K'$. To then apply \cref{MainThmWeakBound1}, we need good estimates for the maximal functions $M^{K'}$ and $R^{K'}$. Since $K$ is odd and homogeneous, we are able to choose the pieces $K'$ to be odd and homogeneous. Then we can write them as superpositions of Hilbert transforms along the curves $(t, c\sgn(t) t^2)$ to obtain bounds for the maximal functions. The oddness assumption can be replaced by the condition $K(x,y) = -K(-x,y)$, in that case the same argument works with Hilbert transforms along parabolas.
    For general kernels $K$ we are not aware of methods to decompose them while keeping good control of the maximal functions.
\end{remark}

\begin{proof}[Proof of \cref{MainThmDini}]
    Fix some function $\eta \in C_c^\infty((-1,1))$ with $\int \eta = 1$ and set $\eta_j(x) = 2^{-j} \eta(2^{-j}x)$ for $j \leq 0$, and $\eta_1(x) = 0$.
    Then we have $\phi = \sum_{j \leq 0} \phi_j$ with $\phi_j = \phi * (\eta_j - \eta_{j + 1})$.
    By the Dini continuity of $\phi$ it holds for all $j$ that
    \[
        |\phi_j(x)| \leq \int |\phi(x-y) - \phi(x)|(|\eta_j(y)| + |\eta_{j+1}(y)|) \, \mathrm{d}y \lesssim \omega(2^j)\,. 
    \]
    The functions $\phi_j$ are differentiable and $|\phi_j'(x)| \lesssim \omega(1) 2^{-2j}$.
    We decompose the kernel $K$: Let
    \[
        K^j(x) = \frac{1}{\rho(x)^{|\alpha|}} \phi_j(\delta_{\rho(x)^{-1}}(x))\,.
    \]
    The uniform convergence of $\phi = \sum_{j \leq 0} \phi_j$ implies that $K = \sum_{j \leq 0} K^j$ with uniform convergence on compact subsets not containing $0$.  Denoting by $(T^j)^\mathcal{Q}$ the maximal modulation operator on the torus associated to $K^j$, we thus have for all $f \in L^2(\mathbb{T}^\mathbf{d})$ that
    \[
        T^\mathcal{Q} f \leq \sum_{j \leq 0} (T^j)^\mathcal{Q} f
    \]
    The estimates for $\phi_j$ and $\phi_j'$ imply that $K^j$ is a Calderón-Zygmund kernel with constant $A(j) \lesssim 2^{-2j}$. Using boundedness of the maximal functions associated to the curves $(t, c\sgn(t)t^2)$, one obtains
    \begin{align*}
        \|M^{K^j}\|_{2\to 2} \sim \|\phi_j\|_1 \lesssim \omega(2^j)\,,\quad \|R^{K^j}\|_{2\to 2} \lesssim \|\phi_j\|_1 \lesssim \omega(2^j)\,.
    \end{align*}
    Applying \cref{MainThmWeakBound1}, it follows that
    \[
        \|(T^j)^\mathcal{Q}\|_{2 \to p} \lesssim  \omega(2^j) \log^2(e + 2^{-2j} \frac{\omega(1)}{\omega(2^j)}) \lesssim j^2 \omega(2^j)\,.
    \]
    We conclude that 
    \[
        \|T^\mathcal{Q}\|_{2 \to p} \leq \sum_{j \leq 0} \|(T^j)^\mathcal{Q}\|_{2 \to p} \lesssim \sum_{j \leq 0} j^2 \omega(2^j) \lesssim \int_0^1 \omega(t)\log^2(\frac{1}{t})\,\frac{\mathrm{d}t}{t} \lesssim \|\omega\|\,. 
    \]
    This completes the proof of the $L^2(\mathbb{T}^2) \to L^p(\mathbb{T}^2)$ estimate for $T^\mathcal{Q}$. The weak type estimates on the torus and on $\R^2$ follow now exactly as in the proof of \cref{MainThmAnisotopic}. 
\end{proof}

\printbibliography

\end{document}